\documentclass[12pt]{amsart}

\usepackage{amsfonts}
\usepackage{amsthm}
\usepackage{amsmath}
\usepackage{cite}

\usepackage{amssymb, latexsym}
\usepackage{graphicx}
\usepackage[left=1.0 in, right=1.0 in, top=1.0 in, bottom=1.0 in]{geometry}
\usepackage[all]{xy}
\usepackage[section]{placeins}
\usepackage{enumitem}

\newtheorem{thm}{Theorem}[section]
\newtheorem{prop}[thm]{Proposition}
\newtheorem{cor}[thm]{Corollary}
\newtheorem{lem}[thm]{Lemma}

\theoremstyle{definition}
\newtheorem{defn}[thm]{Definition}

\theoremstyle{remark}
\newtheorem{remark}[thm]{Remark}
\newtheorem{problem}[thm]{Problem}

\begin{document}

\title{The Local Lifting Problem for $D_4$}
\author{Bradley Weaver}
\address{University of Virginia, 141 Cabell Drive, Charlottesville, VA 22904}
\email{brw4sz@virginia.edu}

\begin{abstract}
For a prime $p$, a cyclic-by-$p$ group $G$ and a $G$-extension $L|K$ of complete discrete valuation fields of characteristic $p$ with algebraically closed residue field, the local lifting problem asks whether the extension $L|K$ lifts to characteristic zero. In this paper, we characterize $D_4$-extensions of fields of characteristic two, determine the ramification breaks of (suitable) $D_4$-extensions of complete discrete valuation fields of characteristic two, and solve the local lifting problem in the affirmative for every $D_4$-extension of complete discrete valuation fields of characteristic two with algebraically closed residue field; that is, we show that $D_4$ is a local Oort group for the prime 2.
\end{abstract}
\maketitle
\section{Introduction}
For a prime $p$, a cyclic-by-$p$ group $G$ and a $G$-extension $L|K$ of complete discrete valuation fields of characteristic $p$ with algebraically closed residue field, the \emph{local lifting problem} (see Problem~\ref{llpgt}) asks whether the extension $L|K$ \emph{lifts to characteristic zero} (a notion whose definition we shall recall in Subsection~\ref{llpss}). The chief aim of this paper is to answer the local lifting problem in the affirmative for the prime $p = 2$ and the group $G = D_4$ in all cases, that is, to show that $D_4$ is a \emph{local Oort group} for $p=2$. 
\let\thefootnote\relax\footnote{\emph{2010 Mathematics Subject Classification}: 14H37, 12F10 Primary; 13B05, 14B12 Secondary.}
\footnote{\emph{Date}: June 23rd, 2017.}

\subsection{The Global Lifting Problem}

The local lifting problem, as stated above, is (upon reformulation) a natural local correlate to the global lifting problem, which may be stated as follows:

\begin{problem}[Global Lifting Problem]\label{glprob}
Suppose that $Y$ is a smooth proper curve over an algebraically closed field $k$ of positive characteristic $p$, and that $\iota: G \to \text{Aut}_k(Y)$ is a faithful action of a finite group $G$ on $Y$ by $k$-automorphisms.
Do there exist a finite integral extension $R$ of the Witt ring $W(k)$, a flat relative curve $\widetilde{Y} \to \text{Spec } R$ and a faithful action $\displaystyle\tilde{\iota}: G \to \text{Aut}_R(\widetilde{Y})$ such that
\begin{enumerate}
\item $\displaystyle\widetilde{Y} \times_R k \cong Y$, and
\item the action $\tilde{\iota}$ on $\widetilde{Y}$ restricts to the action $\iota$ on $Y$?
\end{enumerate}
\end{problem}

\begin{remark} The \emph{Witt ring}, or \emph{ring of Witt vectors}, $W(k)$ of $k$ is the unique complete discrete valuation ring, necessarily of characteristic zero, with uniformizer $p$ and residue field $k$~\cite{LocalFields}.
\end{remark}

If, for a particular $Y$ and $\iota$, the global lifting problem for that curve and action is answered in the affirmative, then we say both that $\iota$ \emph{lifts to characteristic zero} and that $Y$ (with $G$-action $\iota$) \emph{lifts to characteristic zero}. Moreover, we say that $\tilde{\iota}$ and $\widetilde{Y}$ (with $G$-action $\tilde{\iota}$) are, respectively, \emph{lifts} of $\iota$ and of $Y$ (with $G$-action $\iota)$ over $R$.

\begin{defn} A finite group $G$ is an \emph{Oort group} for an algebraically closed field $k$ of characteristic $p$ if every faithful $G$-action on every smooth proper curve over $k$ by $k$-automorphisms lifts to characteristic zero. If $G$ is an Oort group for every algebraically closed field of characteristic $p$, then $G$ is an \emph{Oort group} for the prime $p$.
\end{defn}

The following theorem is a consequence of Grothendieck's results on tame lifting, to wit, of Expos\'{e} XIII, Corollaire 2.12 in \cite{SGA1}, and implies that there is no obstruction to lifting in the tame case. For an exposition, see~\cite{Wew99}.

\begin{thm}[Grothendieck]\label{tamelift} Suppose that $G$ is a finite group with order prime to $p$. Then $G$ is an Oort group for $p$.
\end{thm}

Furthermore, in~\cite{OSS89}, Oort, Sekiguchi and Suwa proved the following:

\begin{thm}\label{plift} For all $m$ such that $p\nmid m$, the group $\mathbb{Z}/pm\mathbb{Z}$ is an Oort group for $p$.
\end{thm}

\subsection{The Local Lifting Problem}\label{llpss}

Let $k$ be an algebraically closed field of positive characteristic $p$, let $Y$ be a smooth proper curve over $k$, and let $\iota: G \to \text{Aut}_k(Y)$ be a faithful action of a finite group $G$ on $Y$ by $k$-automorphisms. For every point $P$ of $Y$, the action $\iota$ induces a faithful action $\iota_P$ by continuous $k$-automorphisms of the inertia group $I_P$ of $G$ at $P$ on the complete local ring of $Y$ at $P$. Since this complete local ring is necessarily isomorphic to a power series ring over $k$ in one variable, the induced action $\iota_P$ prompts the local lifting problem.

\begin{problem}[Local Lifting Problem]
Suppose that a finite group $G$ has a faithful action $\iota: G \to \text{Aut}_{k, \text{cont}}(k[[t]])$ on the power series ring $k[[t]]$ by continuous $k$-automorphisms.
Do there exist a finite integral extension $R$ of the Witt ring $W(k)$ and a faithful action $\tilde{\iota}: G \to \text{Aut}_{R, \text{cont}}(R[[T]])$ on the power series ring $R[[T]]$ such that
\begin{enumerate} 
\item $T$ reduces to $t$ under the canonical map $R \to k$, and
\item the action $\tilde{\iota}$ restricts to the action $\iota$?
\end{enumerate}
\end{problem}

Analogously to the global setting, we say that $\iota$ \emph{lifts to characteristic zero} if such an action $\tilde{\iota}$ exists, and that $\tilde{\iota}$ is a lift of $\iota$.

\begin{defn}
Let $G$ be a finite group. If every faithful $G$-action by continuous $k$-automorphisms on the power series ring $k[[t]]$ lifts to characteristic zero, then $G$ is a \emph{local Oort group} for $k$. If $G$ is a local Oort group for all algebraically closed fields of characteristic $p$, then $G$ is a \emph{local Oort group} for the prime $p$.
\end{defn}

\begin{remark} Any faithful $G$-action by continuous $k$-automorphisms on a power series ring $k[[t]]$ over $k$ induces a Galois extension $k[[t]]^G \to k[[t]]$ of complete discrete valuation rings with Galois group $G$. As shown, \emph{e.g.}, in Chapter IV of~\cite{LocalFields}, the Galois group of any finite Galois extension of complete discrete valuation rings is isomorphic to a cyclic-by-$p$ group, that is, to a group isomorphic to $P \rtimes \mathbb{Z}/m\mathbb{Z}$, where $P$ is a $p$-group and $m$ is prime to $p$. We shall thus, in discussing local Oort groups for $p$, consider only cyclic-by-$p$ groups.
\end{remark}

Using the Galois extension of complete discrete valuation rings induced by a faithful $G$-action by continuous $k$-automorphisms on $k[[t]]$, we may reformulate the local lifting problem as follows:

\begin{problem}[Local Lifting Problem, Galois Theory Reformulation]\label{llpgt}
Let $A$ be a finite Galois extension of $k[[t]]$ with Galois group $G$. Do there exist a finite integral extension $R$ of the Witt ring $W(k)$ and a $G$-Galois extension $\widetilde{A}$ of $R[[T]]$ such that
\begin{enumerate}
\item $\widetilde{A} \otimes_R k \cong A$, and
\item the Galois action on $\widetilde{A}$ over $R[[T]]$ restricts to the Galois action on $A$ over $k[[t]]$?
\end{enumerate}
\end{problem}

If such an $\widetilde{A}$ exists, we say that the extension $A|k[[t]]$ \emph{lifts to characteristic zero}, and, by analogy, that the corresponding extension $\text{Frac}(A)|k((t))$ of complete discrete valuation fields \emph{lifts to characteristic zero}, as well.

The close connection between the global and local lifting problems is manifest in the presence, in this setting, of a local-to-global principle, proven by Garuti in~\cite{Gar96}.

\begin{thm}[Local-to-Global Principle] Let $Y$ be a smooth proper curve over $k$, let $\iota$ be a faithful action of a finite group $G$ on $Y$ by $k$-automorphisms, and let $P_i, 1\le i\le N$, denote the points of $Y$ ramified under $\iota$.  Then $\iota$ lifts to characteristic zero if and only if, for each point $P_i$ of $Y$, the induced action $\iota_{P_i}$ on the complete local ring of $Y$ at $P_i$ lifts to characteristic zero.
\end{thm}

\begin{remark} If $P$ is not a ramification point of $\iota$, that is, if the inertia group of $G$ at $P$ is trivial, then the induced action $\iota_P$ lifts to characteristic zero trivially.
\end{remark}

In~\cite{CGH08}, Chinburg, Guralnick and Harbater proved a close relation between Oort groups and local Oort groups.

\begin{thm}[Theorem 2.4 in~\cite{CGH08}]\label{cgh08thm}
Let $G$ be a finite group. Then $G$ is an Oort group for $k$ if and only if every cyclic-by-$p$ subgroup of $G$ is a local Oort group for $k$.
\end{thm}

Moreover, for cyclic-by-$p$ groups, Oort groups for $k$ and local Oort groups for $k$ coincide.

\begin{thm}[Theorem 2.1 in~\cite{CGH17}]\label{cgh17thm}
Let $G$ be a cyclic-by-$p$ group. Then $G$ is an Oort group for $k$ if and only if $G$ is a local Oort group for $k$.
\end{thm}

We now briefly recall the known results concerning local Oort groups. From Theorems~\ref{tamelift} and~\ref{plift}, any cyclic group of order not divisible by $p^2$ is a local Oort group for $p$.
Moreover, Green and Matignon proved in~\cite{GM98} that, for $m$ such that $p \nmid m$, the group $\mathbb{Z}/p^2m\mathbb{Z}$ is local Oort for $p$,
Bouw and Wewers in~\cite{BoW06} proved for odd $p$ that the dihedral group $D_p$ is local Oort for $p$, and Pagot in~\cite{pagot2002} proved that $D_2 \cong \mathbb{Z}/2\mathbb{Z} \times \mathbb{Z}/2\mathbb{Z}$ is local Oort for 2.

In 2014, Obus and Wewers in~\cite{OW14} and Pop in~\cite{Pop14} jointly resolved the 
\emph{Oort conjecture}, that is, they proved that every finite cyclic group is local Oort for $p$. Finally, Obus has proven, in~\cite{ObusD9} and~\cite{ObusA4}, respectively, that $D_9$ is local Oort for 3, and that $A_4$ is local Oort for 2.

On the other hand, in~\cite{CGH11}, Chinburg, Guralnick and Harbater introduced a particular obstruction to local lifting (denominated the \emph{Katz--Gabber--Bertin obstruction}, or more succinctly, the \emph{KGB obstruction}), and showed that this obstruction prevents all but a few classes of cyclic-by-$p$ groups from being local Oort.

\begin{thm}[Chinburg, Guralnick, Harbater] Suppose that $G$ is local Oort for $k$. Then $G$ is either cyclic or dihedral of order $2p^n$, or \emph{(}for $p=2$\emph{)} either $A_4$ or the generalized quaternion group $Q_{2^m}$ of order $2^m$ for some $m\ge 4$.
\end{thm}

Furthermore, in~\cite{BrW09}, Brewis and Wewers introduced the \emph{Hurwitz tree obstruction}, and showed that this obstruction prevents the generalized quaternion groups from being local Oort. 

Combining all of the foregoing results together, we see that the groups whose status as local Oort groups is open are, save the known local Oort group $D_9$, precisely the dihedral groups of order $2p^n$ for $n > 1$. As noted above, in this paper we shall, as our main aim, remove one further group from this list --- the group $D_4$; in particular, we shall prove that $D_4$ is a local Oort group for  $p = 2$.
It should be noted that $D_4$ differs from $D_9$ in having no tame subextension and from $D_2$ in being non-abelian. To prove that $D_4$ is indeed local Oort, we shall employ the `method of equicharacteristic deformation' used both by Pop in~\cite{Pop14} and by Obus in~\cite{ObusD9} and~\cite{ObusA4}; that is, we shall make equicharacteristic deformations such that the ramification breaks of 
the local extensions on the generic fiber of the deformation are, in a suitable way, smaller than those of the original extension.  Using induction, we shall thus be able to reduce the problem to a particular class of extensions with small ramification breaks, defined by Brewis in~\cite{Brewis2008} as the \emph{supersimple} $D_4$-extensions. Since, in the same paper, Brewis proves that all supersimple $D_4$-extensions in characteristic two lift to characteristic zero, we shall accordingly have completed the desired proof.

\section{Preliminary Definitions and Background}

In this section, we shall introduce a couple of definitions and recall some necessary background information. All of the results in this section are well known; nevertheless, we provide proofs of a few results, as their proofs are somewhat difficult to find in the literature.

\subsection{Higher Ramification Definitions}
Let $k$ be a field of characteristic two, let $K = k((t))$ be the field of Laurent series over $k$, and let $L$ be a finite Galois extension of $K$ such that the residue field of $L$ is separable over $k$. Moreover, let $G$ be the Galois group of $L$ over $K$. For all $j\ge -1$, let $G_j$ and $G^j$ denote, as in~\cite{LocalFields}, the $j$th lower ramification group and the $j$th upper ramification group, respectively, of $G$. Suppose that $G$ is a group of order $2^n$. In this context, we make the following defintions.

\begin{defn}
For all $1\le i\le n$, 
the \emph{$i$th lower ramification break} $\ell_i$ of $G$ is
\[\max \{\nu \mid |G_{\nu}| \ge 2^{n+1-i}\}\]
and, similarly, the \emph{$i$th upper ramification break} $u_i$ of $G$ is
\[\max \{\nu \mid |G^{\nu}| \ge 2^{n+1-i}\}.\]
\end{defn}

\begin{defn}
The \emph{sequence of ramification groups} of $L$ over $K$ is the finite sequence $(G^{u_i})_{i=1}^n$.
\end{defn}

\begin{remark} Since $G_{\ell_i} = G^{u_i}$ for all $i$, the sequence of ramification groups of $L$ over $K$ can also be written as $(G_{\ell_i})_{i=1}^n$.
\end{remark}

For convenience, if $G$ has order two, we shall use the term $\emph{conductor}$ to denote the unique ramification break of $G$. This agrees with the usage of, \emph{e.g.},Bouw and Wewers in~\cite{BoW06}; others, such as Garuti in~\cite{Gar99} define the conductor to be the unique ramification break of $G$ plus one.

\subsection{Artin--Schreier Theory}

Let $K$ be a field of characteristic two, fix an algebraic closure $K^{\text{alg}}$ of $K$, and let $\wp: K^{\text{alg}} \to K^{\text{alg}}$ denote the Artin--Schreier additive group homomorphism, which is given by the assignment
\[ F \mapsto F^2 + F\] on $K^{\text{alg}}$. For the moment we do not insist that $K$ be a complete discrete valuation field. For any element $F$ in $K$, we denote by $[F]$ the image of $F$ in $K/\wp(K)$, and define two elements $F_1$ and $F_2$ of $K$ to be \emph{Artin--Schreier-equivalent} over $K$ if $[F_1] = [F_2]$.  By Artin--Schreier theory, $\wp$ induces a map
\[\Phi : K \to \{L|K \text{ separable} \mid \deg_K (L) = 2\} \cup \{K\}\] given by the assignment $\Phi(F) = K[\wp^{-1}(F)]$ for all $F \in K$.  

\begin{prop}\label{asprop} Let $F_1, F_2 \in K$. Then $[F_1] = [F_2]$ if and only $\Phi(F_1) = \Phi(F_2)$.
\end{prop}

\begin{proof} Suppose $[F_1] = [F_2]$. Then there exists $\alpha \in K$ such that $\alpha^2 + \alpha = F_1 + F_2$. Thus $\wp^{-1}(F_1) + \alpha = \wp^{-1}(F_2)$, and hence $\Phi(F_1) = \Phi(F_2)$.

Now suppose $\Phi(F_1) = \Phi(F_2) \neq K$ (If $\Phi(F_1) = K$, then $[F_1] = [F_2] = 0$.) Let $\alpha_1, \alpha_2 \in \Phi(F_1)$ such that $\wp(\alpha_1) = F_1$ and $\wp(\alpha_2) = F_2$, and let $\sigma$ be the unique non-trivial element of $\text{Gal}(\Phi(F_1)|K)$. Then $\wp(\alpha_1 + \alpha_2) = F_1 + F_2$, and
\[\sigma(\alpha_1 + \alpha_2) = \sigma(\alpha_1) + \sigma(\alpha_2) = (\alpha_1 + 1) + (\alpha_2 + 1) = \alpha_1 + \alpha_2.\]
Hence $\alpha_1 + \alpha_2 \in K$, and $[F_1] = [F_2]$.
\end{proof}

For our purposes it will suffice to consider the case in which $K$ is a complete discrete valuation field, \emph{i.e.}, in which $K = k((t))$ for some field $k$ of characteristic two. Accordingly, we suppose for the remainder of this subsection that $K$ is such a field.

\begin{defn}\label{stdformdef}
An element $\sum_{n \ge - N} a_n t^{n}$ of $K$ is in \emph{standard form over $K$ with respect to $t$} if each coefficient $a_n$ is zero unless $n$ is both negative and odd.
\end{defn}

\begin{prop}\label{stdformuniq}
Suppose that $F_1$ and $F_2$ are distinct standard form elements of $K$. Then $[F_1] \neq [F_2]$.
\end{prop}

\begin{proof}
Since $F_1$ and $F_2$ are distinct, $F_1 + F_2$ is a non-zero standard form element of $K$. Thus the valuation $v_K(F_1 + F_2) = -\deg_{t^{-1}}(F_1 + F_2)$ is odd and negative. Since, for all $\alpha \in K$, the valuation $v_K(\alpha^2 + \alpha) = 2v_K(\alpha)$ if $v_K(\alpha) < 0$, no element of $\wp^{-1}(F_1 + F_2)$ is in $K$. Thus $[F_1 + F_2] \neq 0$; \emph{i.e.}, $[F_1] \neq [F_2]$.
\end{proof}

If the residue field $k$ of $K$ is algebraically closed, then Definition~\ref{stdformdef} obviates one difficulty associated with the equivalence relation defined above --- that, in general, it may not be possible readily to select a canonical element from each Artin--Schreier equivalence class of $K$. In particular, the following propostion holds.

\begin{prop}\label{stdformprop}
Suppose $k$ is algebraically closed. Then every Artin--Schreier equivalence class of $K$ contains precisely one standard form element of $K$.
\end{prop}

\begin{proof}
By Proposition~\ref{stdformuniq}, it suffices to show that every element of $K$ is Artin--Schreier-equivalent over $K$ to a standard form element of $K$.
Let $F = \sum_{n \ge -N} a_n t^n \in K$. For all $n\ge 1$, the equation
\[a_n t^n= \left(\sum_{j \ge 0} a_n^{2^j} t^{2^j n}\right)^2 + \sum_{j \ge 0} a_n^{2^j}t^{2^j n}\]
implies that $[a_n t^n] = 0$. Since $k$ is algebraically closed, $[a_0] = 0$ as well. Thus
\[[F] = \left[\sum_{-N \le n \le -1} a_n t^n \right].\]
Moreover, if $1 \le 2^\ell m \le N$ and $m$ is odd, then
\[ \left[a_{-2^\ell m}t^{-2^\ell m}\right] = \left[\left(a_{-2^\ell m}\right)^{2^{-\ell}}t^{-m}\right],\]
and so $F$ is Artin--Schreier-equivalent over $K$ to a standard form element of $K$.
\end{proof}

\begin{remark}
If $k$ is not algebraically closed, not every Artin--Schreier equivalence class need contain a standard form element. For example, if $k = \mathbb{F}_2$, then $[1] \neq [0]$ over $K$; hence the class $[1]$ contains no standard form element in $K$.
\end{remark}

The conductor of a non-trivial extension associated to an element whose degree in $t^{-1}$ is both negative and odd may be computed from this element as indicated in the following proposition. In particular, the conductor may be computed from any associated non-zero standard form element of $K$.

\begin{prop}\label{ramprop} Let $F \in K$, and let $f = \deg_{t^{-1}} F$. Suppose that $f$ is both negative and odd. Then $\Phi(F) = K[\wp^{-1}(F)]$ is a totally ramified degree two extension of $K$ whose conductor is $f$.
\end{prop}

\begin{proof} Let $\alpha \in \Phi(F)$ such that $\alpha^2 + \alpha = F$. Note that then $v_{\Phi(F)}(F) < 0$ since $v_K(F) = -f < 0$. Since \[v_{\Phi(F)}(F) = \min\{2v_{\Phi(F)}(\alpha), v_{\Phi(F)}(\alpha)\},\]
it follows that $v_{\Phi(F)}(F) = 2 v_{\Phi(F)}(\alpha)$. Thus $v_{\Phi(F)}(F)$ is even. Since $v_K(F) = -f$ is odd, the ramification index of $\Phi(F)$ over $K$ is 2; thus, $\Phi(F)$ is totally ramified over $K$.

To determine the conductor of $\Phi(F)$ over $F$, let $\pi = \alpha t^{(f+1)/2}$, and observe that $v_{\Phi(F)}(\pi) = 1$; \emph{i.e.}, that $\pi$ is a uniformizer of $\Phi(F)$. Let $g(T)$ be the characteristic polynomial of $\pi$ over $K$. Since $\Phi(F)$ is totally ramified over $K$, the different $\mathcal{D}_{\Phi(F)|K}$ of $\Phi(F)$ over $K$ is equal to $g'(\pi)$ by Lemma III.3 and Corollary 2 of Lemma III.2 in~\cite{LocalFields}.
Since $\alpha^2 + \alpha = F$, the relation $\pi^2 + t^{(f+1)/2}\pi = Ft^{f+1}$ holds. Thus \[g(T) = T^2 + t^{(f+1)/2}T + Ft^{f+1},\] and $g'(T) = t^{(f+1)/2}$. Hence $\mathcal{D}_{\Phi(F)|K} = (g'(\pi)) = (t)^{(f+1)/2}$. Since $v_{\Phi(F)}(t) = 2$, the valuation $v_{\Phi(F)}(\mathcal{D}_{\Phi(F)|K}) = f + 1$. By Hilbert's different formula (Proposition IV.4 in~\cite{LocalFields}), it follows that the conductor of $\Phi(F)$ over $K$ is $f$.
\end{proof}

\subsection{Degree of the Different}
Let $K$ be the field of fractions of a discrete valuation ring $A$ with maximal ideal $\mathfrak{m}$, let $L$ be a finite \'{e}tale algebra over $K$, \emph{i.e.}, a finite product of finite separable field extensions of $K$, and let $B$ be the integral closure of $K$ in $L$.

\begin{defn}
Let $\mathfrak{D}_{B|A} = \prod_{i=1}^m \mathfrak{P}_i^{n_i}$ denote the different of $B$ over $A$. Then the \emph{degree of the different} $\delta_{B|A}$ of $B$ over $A$ is the length of $B/\mathfrak{D}_{B|A}$ as a $A/\mathfrak{m}$-module.
\end{defn}

\begin{remark}
This definition agrees with that used in~\cite{GM98},\cite{Brewis2008} and \cite{Obus17arXiv}.
Note that the sum $\sum_{i=1}^n n_i$ does not always give $\delta_{B|A}$, though this is the case if the residue field $B/\mathfrak{P}B$ is equal to $A/\mathfrak{m}A$ for every prime $\mathfrak{P}$ in $B$.
\end{remark}



\section{$D_4$-Extensions as Galois Closures of Non-Galois Extensions}
In this section we shall realize $D_4$-extensions of fields of characteristic two as the Galois closures of (two-level) towers of $\mathbb{Z}/2\mathbb{Z}$-extensions. Throughout the section, let $K$ be a field of characteristic two, let $K^{\text{alg}}$ be a fixed algebraic closure of $K$, let $M \subset K^{\text{alg}}$ be a separable extension of $K$ of degree two, and let $N \subset K^{\text{alg}}$ be a separable extension of $M$ of degree not exceeding two. Note that then there exist $F, G, H \in K$ and $q, r, s \in K^{\text{alg}}$ such that 
\[q^2 + q = F, \quad r^2 + r = Gq + H \quad\text{ and }\quad s^2 + s = G,\]
 and such that $M = K[q]$ and $N = M[r]$. Moreover, there exists $\sigma \in \text{Gal}(K^{\text{alg}}|K)$ such that $\sigma|_M$ is the unique non-trivial element of $\text{Gal}(M|K)$.


\begin{lem}\label{qslem}
The equation
\[(qs)^2 + qs = Gq + Fs^2 = Gq + Fs + FG\]
holds.
\end{lem}

\begin{proof}
Note that
\[(qs)^2 + qs = q^2s^2 + qs^2 + qs^2 + qs = q(s^2 + s) + (q^2 + q)s^2
= Gq + Fs^2 = Gq + Fs + FG.\]
\end{proof}

\begin{lem}\label{kerlem1}
$[G] = 0$ over $M$ if and only if either $[G] = 0$ over $K$ or $[G] = [F]$ over $K$.
\end{lem}

\begin{proof}

Suppose $[G] = 0$ over $M$.  Then there exist $\alpha, \beta \in K$ such that
\begin{align*}G &= (\alpha q + \beta)^2 + \alpha q + \beta
= \alpha^2 q^2 + \beta^2 + \alpha q + \beta \\
&= \alpha^2 (q + F) + \alpha q + \beta^2 + \beta
= (\alpha^2 + \alpha)q + \alpha^2 F + \beta^2 + \beta.
\end{align*}
Since $G \in K$ and $M = K[q]$, it follows that $\alpha^2 + \alpha = 0$. Thus either $\alpha = 0$, in which case $[G] = 0$ over $K$, or $\alpha = 1$, in which case $[G] = [F]$ over $K$.




Now suppose either that $[G] = 0$ over $K$, or that $[G] = [F]$ over $K$.
If $[G] = 0$ over $K$, then $[G] = 0$ over $M$.
If $[G] = [F]$ over $K$, then $[G] = [F] = 0$ over $M$ since $q^2 + q = F$ and $q\in M$.
\end{proof}

\begin{lem}\label{kerlem2}
The following three conditions are equivalent:

\begin{enumerate}[label=\emph{(\arabic*)}]
\item $[Gq + H] = 0$ over $M$.
\item $[G] = 0$ over $K$ and $[H] = [Fs^2]$ over $M$.
\item $[G] = 0$ over $M$ and $[H] = [Fs^2]$ over $M$.
\end{enumerate}
\end{lem}

\begin{proof}

((1)$\implies$(2)) Suppose $[Gq + H] = 0$ over $M$.  Then there exist $\alpha, \beta \in K$ such that
\[Gq + H = (\alpha q + \beta)^2 + \alpha q + \beta = (\alpha^2 + \alpha)q + \alpha^2 F + \beta^2 + \beta,\]
as above. Hence, since $G, H \in K$, it follows that $G = \alpha^2 + \alpha$ and $H = \alpha^2F + \beta^2 + \beta$. Therefore, $[G] = 0$ over $K$, and either $\alpha = s$ or $\alpha = s+1$.

First suppose $\alpha = s$. Then $H = Fs^2 + \beta^2 + \beta$, and hence $[H] = [Fs^2]$ over $K$. Thus $[H] = [Fs^2]$ over $M$ as well.

Now suppose $\alpha = s + 1$. Then 
\[H = (s+1)^2F + \beta^2 + \beta = Fs^2 + F + \beta^2 + \beta,\] and hence $[H] = [Fs^2 + F]$ over $K$. Thus, over $M$, $[H] = [Fs^2 + F] = [Fs^2] + [F] = [Fs^2]$.

\noindent Therefore, in both cases, $[H] = [Fs^2]$ over $M$. Thus $[H] = [Fs^2]$ over $M$.

((2)$\implies$(3)) Since $K\subseteq M$, this implication holds \emph{a fortiori}.

((3)$\implies$(1)) Finally, suppose that $[G] = 0$ over $M$ and that $[H] = [Fs^2]$ over $M$. By Lemma~\ref{qslem}, $(qs)^2 + qs = Gq + Fs^2$. Since $[G] = 0$ over $M$, it follows that $s \in M$ and that $qs \in M = K[q]$. Thus, over $M$, 
\[0 = [(qs)^2 + qs] = [Gq + Fs^2] = [Gq] + [Fs^2] = [Gq] + [H] = [Gq + H].\qedhere\]
\end{proof}


\begin{lem}\label{gallem}
Suppose that $N$ is a degree four extension of $K$.
The following four conditions are equivalent:

\begin{enumerate}[label=\emph{(\arabic*)}]
\item $N$ is a Galois extension of $K$.
\item $\sigma(N) = N$.
\item $[\sigma(Gq + H)] = [Gq + H]$ over $M$.
\item $[G] = 0$ over $M$.
\end{enumerate}
\end{lem}

\begin{figure}
\[\xymatrix{ & \sigma(N) \ar@{-}[dr]_{\sigma(Gq+H)} & & N \ar@{-}[dl]^{Gq+H} \\
& & M \ar@{-}[d]^{F} & & \\
& & K & & }\]
\caption{}
\label{galfig}
\end{figure}

\begin{remark} The situation described in Lemma~\ref{gallem} may be visualized as in Figure~\ref{galfig}.
\end{remark}

\begin{proof}
Recall that $\sigma|_M$ is the unique non-trivial element of $M$.


Suppose that $\sigma(N) = N$. Since $\sigma|_M$ is non-trivial, $\sigma|_N$ is non-trivial. 
Let $\tau$ be the unique non-trivial element of $\text{Gal}(N|M)$. 
Then $\tau(N) = N$, and $\tau|_M$ is trivial. Thus $\sigma|_M \neq \tau|_M$; as such, $\sigma|_N$ and $\tau$ are distinct non-trivial $K$-automorphisms of $N$. Hence $N$ is Galois over $K$,
and conditions (1) and (2) are equivalent.

Moreover, since $\sigma|_M$ is non-trivial, $\sigma(q) = q+1$. Thus $\sigma(Gq + H) = G(q+1) + H = Gq + G + H$. Therefore,
$[\sigma(Gq + H)] = [Gq + H]$ over $M$ if and only if $[Gq + G + H] = [Gq + H]$ over $M$, which holds if and only if $[G] = 0$ over $M$. Thus (3) and (4) are equivalent.

Note now that $\left(\sigma(r)\right)^2 + \sigma(r) = \sigma(r^2 + r) = \sigma(Gq + H)$. 
Thus $[\sigma(Gq + H)] = [Gq + H]$ over $M$ if and only if $M[\sigma(r)] = M[r] = N$, which holds if and only if $\sigma(N) = N$. Hence (2) and (3) are equivalent.

Therefore, conditions (1) through (4) are equivalent, as claimed.
\end{proof}

\begin{prop}\label{classprop}
The following statements, exactly one of which applies, all hold: 
\begin{enumerate}[label=\emph{(\arabic*)}]
\item If $[G] = 0$ over $K$ and $[H] = [Fs^2]$ over $M$, then $N = M$.
\item If $[G] = 0$ over $K$ and $[H] \neq [Fs^2]$ over $M$, then $N$ is a Galois extension of $K$, and $\emph{Gal}(N|K) \cong \mathbb{Z}/2\mathbb{Z} \times \mathbb{Z}/2\mathbb{Z}$.
\item If $[G] = [F]$ over $K$, then $N$ is a Galois extension of $K$, and $\emph{Gal}(N|K) \cong \mathbb{Z}/4\mathbb{Z}$.
\item If $[G] \neq 0$ over $M$, then $N$ is not a Galois extension of $K$, and $\emph{Gal}(\tilde{N}|K) \cong D_4$, where $\tilde{N}$ denotes the Galois closure of $N$ over $K$.
\end{enumerate}

\end{prop}

\begin{proof} 

To prove (1), suppose that $[G] = 0$ over $K$ and that $[H] = [Fs^2]$ over $M$. Then, by Lemma~\ref{kerlem2}, $[Gq + H] = 0$ over $M$. Thus $r \in M$, and hence $N = M[r] = M$.


To prove (2), suppose that $[G] = 0$ over $K$ and that $[H] \neq [Fs^2]$ over $M$. By Lemma~\ref{kerlem2}, $[Gq + H] \neq 0$ over $M$; as such, $N \neq M$. Thus $N$ is a degree four extension of $K$. Since $[G] = 0$ over $M$, Lemma~\ref{gallem} implies that $N$ is a Galois extension of $K$. Moreover, since $[G] = 0$ over $K$, it follows that $s \in K$. Thus
\[(r + qs)^2 + (r + qs) = r^2 + r + (qs)^2 + qs = Gq + H + Gq + Fs^2 = H + Fs^2 \in K,\]
where the second equality follows by Lemma~\ref{qslem}.
Since $[H] \neq [Fs^2]$ over $M$, $[H + Fs^2] \neq 0$ over $M$. By Lemma~\ref{kerlem1}, it follows that $[H + Fs^2] \neq 0$ over $K$ and that $[H + Fs^2] \neq [F]$ over $K$. Hence $K[r+qs]$ is a degree two subfield of $N$ that is not equal to $M = K[q]$. Thus $\text{Gal}(N|K) \cong \mathbb{Z}/2\mathbb{Z} \times \mathbb{Z}/2\mathbb{Z}$.

To prove (3), suppose that $[G] = [F]$ over $K$. Then $[G] \neq 0$ over $K$, and so $[Gq + H] \neq 0$ over $M$ by Lemma~\ref{kerlem2}. Thus $N$ is a degree four extension of $K$. Since $[G] = [F] = 0$ over $M$, Lemma~\ref{gallem} implies that $N$ is a Galois extension of $K$. Moreover, since $[G] = [F] \neq 0$ over $K$, it follows that $s\in M\backslash K$, and hence that $\sigma(s) = s + 1$. Furthermore, since
\[(\sigma(r))^2 + (\sigma(r)) = \sigma(Gq + H) = Gq + H + G = (r+s)^2 + (r+s),\]
either $\sigma(r) = r + s$, or $\sigma(r) = r + s + 1$. In either case, one easily verifies that $\sigma^2(r) = r + 1$.
Therefore $\sigma^2|_N$ is not trivial, and $\text{Gal}(N|K) \cong \mathbb{Z}/4\mathbb{Z}$.

To prove (4), suppose that $[G] \neq 0$ over $M$. Then $[G] \neq 0$ over $K$, and hence $N \neq M$ by Lemma~\ref{kerlem2}. Thus $N$ is a degree four extension of $K$; hence, by Lemma~\ref{gallem}, $N$ is not a Galois extension of $K$. Moreover, $\text{Gal}(\tilde{N}|K)$ is isomorphic to a subgroup of $S_4$ and contains an index two (normal) subgroup, \emph{viz}.~$\text{Gal}(\tilde{N}|M)$, which itself contains a subgroup of index four in $\text{Gal}(\tilde{N}|K)$ that is not normal in $\text{Gal}(\tilde{N}|K)$. The only group satisfying all these conditions is $D_4$, so $\text{Gal}(\tilde{N}|K) \cong D_4$.
\end{proof}

\begin{prop}\label{nongalprop}
Let $F_0, G_0, H_0 \in K$, and let $q_0, r_0, s_0 \in K^{\emph{alg}}$ such that $q_0^2 + q_0 = F_0$, $r_0^2 + r_0 = G_0 q_0 + H_0$, and $s_0^2 + s_0 = G_0$.
Also, let $M_0 = K_0[q_0]$ and $N_0 = M_0[r_0]$ . Then

\begin{enumerate}[label=\emph{(\arabic*)}]
\item $M_0 = M$ and $N_0 = N$ if and only if $[F] = [F_0]$ over $K$, $[G] = [G_0]$ over $K$, and $[H] = [H_0 + G_0(q+q_0) + F(s+s_0)^2]$ over $M$.
\item $M_0 = M$ and $N_0 = \sigma(N)$ if and only if $[F] = [F_0]$ over $K$, $[G] = [G_0]$ over $K$, and $[H] = [H_0 + G_0(q+q_0) + F(s+s_0)^2 + G]$ over $M$.
\end{enumerate}
\end{prop}

\begin{proof}
Note that $M_0 = M$ if and only if $[F] = [F_0]$ over $K$. This will be used without citation henceforth.

To prove (1), suppose that $M_0 = M$. Then $N_0 = N$ if and only if $[Gq + H] = [G_0 q_0 + H_0]$ over $M$, which holds if and only if $[(G+G_0)q + G_0(q + q_0) + H + H_0] = 0$ over $M$. Since $[F] = [F_0]$ over $K$, the element $q+q_0$ is in $K$.
Thus, by Lemma~\ref{kerlem2}, $N_0 = N$ if and only if both $[G + G_0] = 0$ over $K$, and $[G_0(q + q_0) + H + H_0] = [F(s+s_0)^2]$ over $M$; \emph{i.e.}, if and only if both $[G] = [G_0]$ over $K$, and $[H] = [H_0 + G_0(q+q_0) + F(s+s_0)^2]$ over $M$. Statement (1) now follows.

To prove (2), suppose that $M_0 = M$, and note both that $\sigma(N) = M[\sigma(r)]$, and that $\left(\sigma(r)\right)^2 + \sigma(r) = \sigma(Gq + H) = Gq + H + G$. 
Then $N_0 = \sigma(N)$ if and only if $[Gq + H + G] = [G_0 q_0 + H_0]$ over $M$, which holds if and only if $[(G+G_0)q + G_0(q+q_0) + H + G + H_0] = 0$ over $M$. As above, by Lemma~\ref{kerlem2}, this holds if and only if $[G] = [G_0]$ over $K$ and $[H] = [H_0 + G_0(q+q_0) + F(s+s_0)^2 + G]$ over $M$. 
Statement (2) now follows.
\end{proof}

\section{$D_4$-Extensions of Fields of Characteristic Two}

Let $K$ be a field of characteristic two, let $K^{\text{alg}}$ be a fixed algebraic closure of $K$, and let $L\subseteq K^{\text{alg}}$ be a Galois extension of $K$ such that $\text{Gal}(L|K) \cong D_4$.

\begin{prop}\label{d4exist}
There exist $F, G, H \in K$, and $q,r \in K^{\emph{alg}}$ such that $q^2 + q = F$, $r^2 + r = Gq + H$, and $L$ is the Galois closure over $K$ of $K[q,r]$.
\end{prop}

\begin{proof}
Note that $D_4$ contains a subgroup of index two containing a non-normal subgroup of index four. Thus there exists a non-normal degree four subfield $L'$ of $L$ over $K$ containing a degree two subfield $K'$ of $L$. Then there exist $F \in K$ and $q\in K^{\text{alg}}$ such that $q^2 + q = F$ and $K[q] = K'$. Hence there exist $G, H \in K$ and $r \in K^{\text{alg}}$ such that $r^2 + r = Gq + H$ and $L' = K'[r] = K[q,r]$. Since $L'\subset L$ is not Galois over $K$, it follows that $L$ is the Galois closure of $L' = K[q,r]$ over $K$, as desired.
\end{proof}

\begin{prop}\label{d4struc}
Suppose that $F, G, H \in K$, and $q,r,s \in K^{\emph{alg}}$ such that $q^2 + q = F$, $r^2 + r = Gq + H$ and $s^2 + s = G$, and $L$ is the Galois closure over $K$ of $K[q,r]$. Then
\begin{enumerate}[label=\emph{(\arabic*)}]
\item the degree two subfields of $L$ are $K[q], K[s]$ and $K[q+s]$,
\item the unique degree four normal subfield of $L$ is $K[q,s]$,
\item the non-normal degree four subfields of $L$ containing $K[q]$ are $K[q,r]$ and $K[q, r+s]$,
\item the non-normal degree four subfields of $L$ containing $K[s]$ are $K[s,r+qs]$ and $K[s,r+qs+q]$, and
\item $L = K[q,r,s]$.
\end{enumerate}
\end{prop}

\begin{figure}
\[\xymatrix{ & & L=K[q,r,s] 
\ar@{-}[dll] \ar@{-}[dl] \ar@{-}[d] \ar@{-}[dr] \ar@{-}[drr] & & \\
K[s, r+qs] \ar@{-}[dr]_{Fs+H+FG} & K[s, r + qs + q] \ar@{-}[d] & 
K[q,s] \ar@{-}[dl] \ar@{-}[d] \ar@{-}[dr] & 
K[q, r + s] \ar@{-}[d] & K[q, r] \ar@{-}[dl]^{Gq+H} \\
& K[s] \ar@{-}[dr]_G & K[q+s] \ar@{-}[d]^{F+G} & K[q] \ar@{-}[dl]^F & \\
& & K & & }\]
\caption{Subfields of $L$ over $K$}
\label{d4fig}
\end{figure}
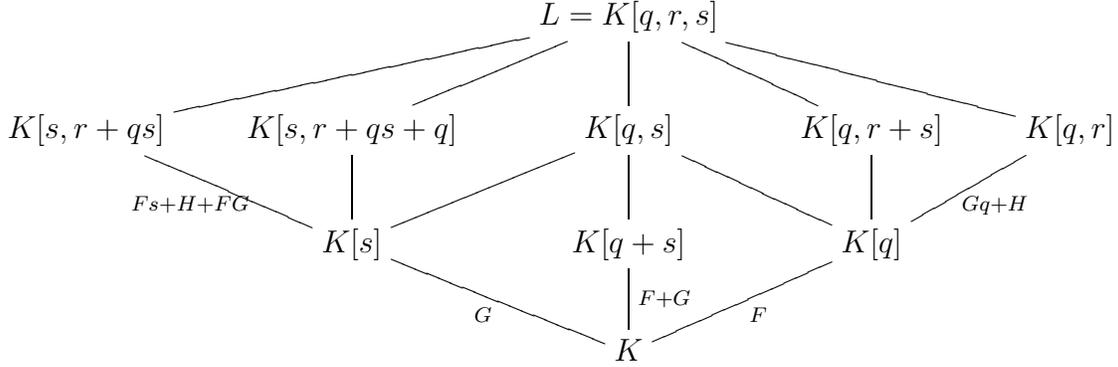

\begin{remark} The situation described in Propostion~\ref{d4struc} may be visualized as in Figure~\ref{d4fig}.
\end{remark}

\begin{proof}
Let $\sigma\in \text{Gal}(L|K)$ such that $\sigma|_{K[q]}$ is non-trivial. Then the non-normal degree four subfields of $L$ containing $K[q]$ are $K[q,r]$ and $\sigma(K[q,r])$. Since $(r+s)^2 + (r+s) = Gq + H + G = \sigma(Gq+H)$, it follows that $\sigma(K[q,r]) = K[q,r+s]$. Statement (3) now follows immediately. 

To prove (1), (2) and (5), note that, since $K[q,r]$ and $K[q,r+s]$ are both subfields of $L$, it follows that $s = r + (r+s) \in L$. Moreover, since $K[q,r]$ is not Galois over $K$, $[G] \neq 0$ over $K[q]$ by Lemma~\ref{gallem}. Hence $K[q,s] = K[q][s]$ is a degree four extension of $K$. Since $\sigma(G) = G$, it follows by Lemma~\ref{gallem} that $K[q,s]$ is a Galois extension of $K$. Statement (2) now follows, and, since $K[q]$, $K[s]$, and $K[q+s]$ are the three degree two subfields of $K[q,s]$, so does (1). Finally, since $K[q,s]$ and $K[q,r]$ are distinct degree four subfields of $L$, $L = K[q,r,s]$; \emph{i.e.}, (5) holds.

To prove (4), recall that $(qs)^2 + qs = Gq + Fs + FG$ by Lemma~\ref{qslem}.
Thus $[Gq + H] = [Fs + H + FG]$ over $K[q,s]$. Since $[G] \neq 0$ over $K[q]$, $[F] \neq 0$ over $K[s]$. As such, since $Fs + H + FG = (r+qs)^2 + (r+qs)$, it follows that $K[s, r+qs]$ is a non-Galois degree four subfield of $L$ by Proposition~\ref{classprop}.
Hence the non-normal degree four subfields of $L$ containing $K[s]$ are $K[s,r+qs]$ and $\tau(K[s,r+qs])$, where $\tau\in\text{Gal}(L|K)$ such that $\tau|_{K[s]}$ is non-trivial. Since
\[(r+qs+q)^2 + (r+qs+q) = Fs + H + FG + F = \tau(Fs+H+FG),\]
it follows that $\tau(K[s,r+qs]) = K[s,r+qs+q]$. Statement (4) now follows immediately.
\end{proof}

\section{$D_4$-Extensions of Complete Discrete Valuation Fields of Characteristic Two}

Having described the structure of $D_4$-extensions over all fields of characteristic two, in this section we shall restrict our attention to complete discrete valuation fields of characteristic two and shall determine (for suitably nice extensions) formulas for the ramification breaks of the given extension.  We shall, moreover, parametrize and classify $D_4$-extensions of complete discrete valuation fields of characteristic two with algebraically closed residue field. To this end, let $k$ be a (not necessarily algebraically closed) field of characteristic two, let $K = k((t))$ be the field of Laurent series over $k$, let $K^{\text{alg}}$ denote a fixed algebraic closure of $K$, and let $L\subseteq K^{\text{alg}}$ be a Galois extension of $K$ such that $\text{Gal}(L|K) \le D_4$.

\begin{defn} The Galois extension $L$ over $K$ is \emph{generated by standard form elements} if there exist $F, G, H \in K = k((t))$ in standard form with respect to $t$ and $q, r ,s \in K^{\text{alg}}$ such that $q^2 + q = F, r^2 + r = Gq + H$ and $s^2 + s = G$, and such that $L = K[q,r,s]$.


\end{defn}

\begin{remark}\label{stdrmk}
By Proposition~\ref{classprop}, $\text{Gal}(L|K) \cong D_4$ unless $F = 0$, $G = 0$, or $F = G$. If $\text{Gal}(L|K) \cong D_4$, we say that $L$ is a \emph{$D_4$-standard form} extension of $K$.
\end{remark}

The triple $(F,G,H)$ may be considered a sort of `canonical form' for a $D_4$-standard form extension of $K$, though any given $D_4$-standard form extension is associated not to one triple, but to several.

\subsection{Parametrization of $D_4$-Extensions via Standard Form Elements}

Suppose now that $k$ is algebraically closed and that $\text{Gal}(L|K) \cong D_4$.

\begin{prop}\label{existprop}
There exist $F, G, H \in K = k((t))$ in standard form with respect to $t$, and $q, r \in K^{\emph{alg}}$ such that $q^2 + q = F, r^2 + r = Gq + H$,
 and $L$ is the Galois closure over $K$ of $K[q,r]$.
\end{prop}

\begin{proof} By Proposition~\ref{d4exist}, there exist $F', G', H' \in K$, not necessarily in standard form, and $q', r', s' \in K^{\text{alg}}$ such that $(q')^2 + q' = F'$, $(r')^2 + r' = G'q' + H'$, and $(s')^2 + s' = G'$, and such that $L$ is the Galois closure over $K$ of $K[q',r']$. Since $k$ is algebraically closed, by Proposition~\ref{stdformprop} there exist unique elements $F,G\in K = k((t))$ in standard form such that $[F] = [F']$ and $[G] = [G']$, respectively, over $K$. Let $q, s\in K^{\text{alg}}$ such that $q^2 + q = F$ and $s^2 + s = G$.
Since 
\[(q+q')^2 + (q+q') = \left(q^2 + q\right) + \left((q')^2 + q'\right) = F + F',\]
and since $[F] = [F']$ over $K$, it follows that $q + q' \in K$. Similarly, $s + s' \in K$, and thus $H' + G'(q+q') + F(s+s')^2 + G \in K$ as well.
Therefore, there exists $H\in K$ in standard form such that $[H] = [H' + G'(q+q') + F(s+s')^2]$ over $K[q] = K[q']$. Let $r\in K^{\text{alg}}$ such that $r^2 + r = Gq +H$. 
Then $K[q,r] = K[q', r']$ by Proposition~\ref{nongalprop}.
\end{proof}


Note that, by Proposition~\ref{d4struc}, $L = K[q,r,s]$. Proposition~\ref{d4exist} thus has the following corollary:

\begin{cor}\label{stdformcor}
The extension $L$ is a $D_4$-standard form extension of $K$.
\end{cor}

As noted above, the `standard form' triple $(F,G,H)$ is not unique; indeed, in this case any given $D_4$-extension of $K$ is associated to eight distinct triples, which are enumerated in the following proposition.

\begin{prop}\label{triplesprop}
Let $L_0 \subseteq K^{\emph{alg}}$ be another Galois extension of $K$ such that $\emph{Gal}(L_0|K) \cong D_4$, and let $F_0, G_0, H_0 \in K$ in standard form such that $L_0$ is the Galois closure of $K[q_0, r_0]$, where $q_0, r_0 \in K^{\emph{alg}}$ such that $q_0^2 + q_0 = F_0$ and $r_0^2 + r_0 = Gq_0 + H_0$. Then the fields $L_0$ and $L$ are equal if and only if one of the four following conditions holds.
\begin{enumerate}[label=\emph{(\arabic*)}]
\item $F_0 = F$, $G_0 = G$, and $[H_0] = [H]$ over $K[q_0]$.
\item $F_0 = F$, $G_0 = G$, and $[H_0] = [H + G]$ over $K[q_0]$
\item $F_0 = G$, $G_0 = F$, and $[H_0] = [H+FG]$ over $K[q_0]$.
\item $F_0 = G$, $G_0 = F$, and $[H_0] = [H+FG+F]$ over $K[q_0]$.
\end{enumerate}
\end{prop}

\begin{proof}
Note from Propositon~\ref{d4struc} that $L_0$ and $L$ are equal if and only if $K[q_0, r_0]$ is equal to one of the four non-normal degree four subfields $K[q,r], K[q,r+s], K[s,r+qs], K[s,r+qs+q]$ of $L$.

By Proposition~\ref{nongalprop}, $K[q_0, r_0] = K[q,r]$ if and only if condition (1) holds, and $K[q_0, r_0] = K[q,r+s]$ if and only if condition (2) holds.
Similarly, since $(r+qs)^2 = Fs + H + FG$ by Lemma~\ref{qslem}, $K[q_0, r_0] = K[s,r+qs]$ if and only if condition (3) holds, and $K[q_0, r_0] = K[s,r+qs+q]$ if and only if condition (4) holds.
\end{proof}


\begin{cor}\label{triplescor}
Let $\mathcal{K}$ be the set of standard form elements of $K$, and let $\mathcal{G}$ be the set of Galois extensions of $K$ contained in $K^{\emph{alg}}$ whose Galois group over $K$ is isomorphic to $D_4$. Furthermore, let $\mathcal{D} = \{(\phi, \gamma, \eta) \in \mathcal{K}^3 \mid \phi = 0 \text{ or }\gamma = 0 \text{ or }\gamma = \phi\}$, and define $\Phi: \mathcal{K}^3\backslash \mathcal{D}\to \mathcal{G}$ such that, for all $(\phi, \gamma, \eta) \in \mathcal{K}^3 \backslash \mathcal{D}$, $\Phi(\phi,\gamma,\eta)$ is the Galois closure of $K[\kappa, \rho]$, where $\kappa, \rho \in K^{\emph{alg}}$ such that $\kappa^2 + \kappa = \phi$ and $\rho^2 + \rho = \gamma\kappa + \eta$. Then $\Phi$ is surjective.
\end{cor}

\begin{remark}\label{eightrmk}
By Lemma~\ref{kerlem1}, each condition in Proposition~\ref{triplesprop} corresponds to two pre-images under $\Phi$ of any given element of $\mathcal{G}$. Thus the surjection $\Phi$ is, in fact, eight-to-one.
\end{remark}

\subsection{Computation of Ramification Breaks}\label{rambrkss}

In this subsection, we shall neither suppose that the residue field $k$ of $K$ is algebraically closed, nor insist that $L$ be generated by standard form elements.  However, we shall suppose that $L$ is a totally ramified extension of $K$, and that $L$ is a \emph{$D_4$-odd form} extension of $K$, as defined below.

\begin{defn}\label{odddef} The Galois extension $L$ over $K$ is a \emph{$D_4$-odd form} extension of $K$ if, firstly, $\text{Gal}(L|K) \cong D_4$ and, secondly, there exist $F, G, H \in K = k((t))$ and $q, r ,s \in K^{\text{alg}}$ such that 
\begin{enumerate}[label=(\arabic*)]
\item $q^2 + q = F, r^2 + r = Gq + H$ and $s^2 + s = G$,
\item $L = K[q,r,s]$,
\item each of $\deg_{t^{-1}} F$, $\deg_{t^{-1}} G$ and $\deg_{t^{-1}}(F + G)$ is both positive and odd, and
\item $\deg_{t^{-1}} H$ is not both positive and even.
\end{enumerate}
\end{defn}

\begin{remark}\label{totramrmk}
Since $K$ is a complete discrete valuation field, a finite extension $M$ of $K$ is totally ramified if and only if the residue field of $M$ is equal to $k$, the residue field of $K$.
\end{remark}

Let $F, G, H, q, r$ and $s$ be as in Definition~\ref{odddef}, and let $f = \deg_{t^{-1}}(F)$, $g = \deg_{t^{-1}}(G), h = \deg_{t^{-1}}(H)$ and $d = \deg_{t^{-1}}(F+G)$. Suppose also that $f\le g$.


The degrees $d,f,g$ and $h$ suffice to determine the lower and upper ramification breaks of the extension $L$ of $K$. In demonstrating this, the following two lemmas, both adapted from Lemme~1.1.4 in~\cite{Ray99}, will be helpful.

\begin{lem}\label{raylem1}
Let $K_0 = k((t_0))$, let $K_1$ and $K_2$ be distinct Artin--Schreier extensions of $K_0$, and let $K_3$ be the unique degree two subfield of $K_1K_2$ distinct from both $K_1$ and $K_2$.
Moreover, let $c_1$, $c_2$ and $c_3$ denote the conductors over $K_0$ of $K_1$, $K_2$ and $K_3$, respectively. Suppose that $K_1K_2$ is totally ramified over $K_0$, and that $c_1 < c_2$. Then $c_3 = c_2$, the conductor of $K_1K_2$ over $K_1$ is $2c_2 - c_1$, and the conductor of $K_1K_2$ over $K_2$ is $c_1$.
\end{lem}

\begin{proof}
Note that, since $K_1$ and $K_2$ are distinct, $K_1K_2$ is an Artin--Schreier extension both of $K_1$ and of $K_2$.

Let $\Gamma$ be the Galois group of $K_1K_2$ over $K$, and let $H_1$, $H_2$ and $H_3$ denote the subgroups of $\Gamma$ consisting of those elements of $\Gamma$ fixing $K_1$, $K_2$ and $K_3$, respectively.
By Proposition~IV.14 in~\cite{LocalFields},
$(\Gamma/H_j)^i = \Gamma^i H_j/H_j$ for all $i\ge -1, j \in \{1,2,3\}$. Thus each of the conductors $c_j$ of $\Gamma/H_j$ is an upper ramification break of $K_1K_2$ over $K_0$. Since $c_1 < c_2$, the sequence of upper ramification breaks of $K_1K_2$ over $K_0$ is $(c_1, c_2)$. Hence $\Gamma^i H_1/H_1 = (\Gamma/H_1)^i = 1$ for all $i > c_1$, and $\Gamma^i = H_1$ for all $c_1 < i \le c_2$. Thus $(\Gamma/H_3)^i = \Gamma^i H_3/H_3 = \Gamma/H_3$ for all $i \le c_2$; as such, $c_3 = c_2$.

To determine the conductors of $K_1K_2$ over $K_1$ and $K_2$, consider the corresponding sequence of lower ramification breaks of $K_1K_2$ over $K_0$. By Herbrand's Formula (see Section IV.3 in~\cite{LocalFields}), this sequence is 
$(c_1, 2c_2 - c_1)$. Thus $\Gamma_i = H_1$ for all $c_1 < i \le 2c_2 - c_1$. Since, by Proposition~IV.2 in ~\cite{LocalFields}, $(H_j)_i = \Gamma_i \cap H_j$ for all $i\ge -1, j \in \{1,2\}$, it follows that
\[(H_1)_i = \Gamma_i \cap H_1 = \begin{cases} H_1 & \text{if }i \le 2c_2 - c_1\\ 1 & \text{if } i > 2c_2 - c_1 \end{cases}\quad\text{and}\quad(H_2)_i = \Gamma_i \cap H_2 = \begin{cases} H_2 & \text{if }i \le c_1\\ 1 & \text{if } i > c_1 \end{cases};\]
\emph{i.e.}, that the conductor of $K_1K_2$ over $K_1$ is $2c_2 - c_1$, and  the conductor of $K_1K_2$ over $K_2$ is $c_1$.
\end{proof}

\begin{lem}\label{raylem2}
Let $K_0 = k((t_0))$, let $K_1$ and $K_2$ be distinct Artin--Schreier extensions of $K_0$, and let $c_1$ and $c_2$ denote the conductors over $K_0$ of $K_1$ and $K_2$, respectively. Moreover, let $K_3$ be the unique degree two subfield of $K_1K_2$ distinct from both $K_1$ and $K_2$, and let $c_3$ be the conductor of $K_3$ over $K_0$. Suppose that $K_1K_2$ is totally ramified over $K_0$, and that $c_1 = c_2$. Then $c_3 \le c_1$, and both the conductor of $K_1K_2$ over $K_1$ and the conductor of $K_1K_2$ over $K_2$ are equal to $c_3$.
\end{lem}

\begin{proof}
Let $\Gamma$ be the Galois group of $K_1K_2$ over $K$, and let $H_1$, $H_2$ and $H_3$ denote the subgroups of $\Gamma$ consisting of those elements of $\Gamma$ fixing $K_1$, $K_2$ and $K_3$, respectively.

If $c_3 > c_1 = c_2$, then, as in Lemma~\ref{raylem1}, $\Gamma^i = H_1$ and $\Gamma^i = H_2$ for all $c_1 = c_2 < i \le c_3$. Since $H_1 \neq H_2$, it follows that $c_3 \le c_1$. Thus $\Gamma^i = \Gamma$ for all $i\le c_3$, and $\Gamma^i \subseteq H_3$ for all $i > c_3$. Hence, for both $j=1$ and $j=2$, $(H_j)_i = \Gamma_i \cap H_j = \Gamma^i \cap H_j = H_j$ for all $i \le c_3$ and $(H_j)_i \subseteq H_j \cap H_3$ is trivial for all $i > c_3$; \emph{i.e.}, both the conductor of $K_1K_2$ over $K_1$ and the conductor of $K_1K_2$ over $K_2$ are equal to $c_3$.
\end{proof}

\begin{lem}\label{condlem}
The conductors of $K[q], K[s]$ and $K[q+s]$ over $K$ are $f$, $g$ and $d$, respectively. Moreover, the conductor of $K[q,r]$ over $K[q]$ is $2\max\{f+g, h\} - f$.
\end{lem}

\begin{proof}
Since the degrees in $t^{-1}$ of $F$, of $G$ and of $F+G$ are all both odd and positive by hypothesis, the first three claims follow immediately by Proposition~\ref{ramprop}.
For the fourth claim, note that $v_q(F) = -2f$ since $K[q]$ is a totally ramified extension of $K$. Thus $v_q(q) = -f$, and $v_q(Gq) = -(2g+f)$.

Suppose $h\le f$. Then $v_q(H) = -2h \ge -2f > -(2g+f) = v_q(Gq)$ since $h \le f \le g$. Hence $v_q(Gq+H) = -(2g+f)$; since $2g+f$ is odd, it follows by Proposition~\ref{ramprop} that the conductor of $K[q,r]$ over $K[q]$ is $2g+f = 2\max\{f+g,h\} - f$.

Now suppose $f < h$. Let $u\in K^{\text{alg}}$ such that $u^2 + u = H$, let $c_u$ denote the conductor of $K[q,u]$ over $K[q]$, and let $C_u$ denote the conductor of $K[q,r+u]$ over $K[q]$.
Since $v_q(Gq+H) = -(2g+f)$, and $2g+f$ is odd, $C_u = 2g + f$ by Proposition~\ref{ramprop}.
Similarly, the conductor of $K[u]$ over $K$ is $h$.
Thus $c_u = 2h-f$ by Lemma~\ref{raylem1} applied to the extensions $K[q]|K$ and $K[u]|K$. Since $f$, $g$ and $h$ are all odd, $2g+f - (2h-f) = 2(f+g-h) \neq 0$; hence $C_u = 2g+f \neq 2h- f = c_u$.
Thus the conductor of $K[q,r]$ over $K[q]$ is $\max\{2g+f, 2h-f\} = 2\max\{f+g,h\} - f$ by Lemma~\ref{raylem1} applied to the extensions $K[q, r+u]|K[q]$ and $K[q, r]|K[q]$.
\end{proof}

\begin{figure}
\[\xymatrix{K[q,r] \ar@{-}[dr]_{2\max\{f+g, h\} - f} & K[q, r+u] \ar@{-}[d]_{C_u} & K[q,u] \ar@{-}[dl]_{c_u} \ar@{-}[d] \ar@{-}[dr] & \\
& K[q] \ar@{-}[dr]_{f} & K[q+u] \ar@{-}[d] & 
K[u] \ar@{-}[dl]^{h} \\
& & K & }\]
\caption{}
\label{condfig}
\end{figure}

\begin{remark} The situation described in the final paragraph of the proof of Lemma~\ref{condlem} may be visualized as in Figure~\ref{condfig}.
\end{remark}

\begin{remark} In the final paragraph of the proof of Lemma~\ref{condlem}, we tacitly assumed that $K[q,r,u]$ was totally ramified over $K$. However, since $K[q,r]$ is totally ramified over $K[q]$, and the conductor of a totally ramified extension of a complete discrete valuation field of characteristic two is invariant under base change, the lemma holds as stated.
\end{remark}

\begin{remark}
So long as $K[q,r]$ is totally ramified over $K$, the fourth claim of Lemma~\ref{condlem} holds even if $d$ is not both odd and positive. In particular, Lemma~\ref{condlem} has the following corollary, which also (essentially) follows from a known result (see, \emph{e.g.},~\cite{Gar99}) on ramification breaks of Witt vectors.
\end{remark}

\begin{cor}\label{condcor}
Suppose that $F = G$ \emph{(}so that $K[q,r]$ is a $\mathbb{Z}/4\mathbb{Z}$-extension of $K$ by Proposition~\ref{classprop}\emph{)}. Then the conductor of $K[q,r]$ over $K[q]$ is $2\max\{2f, h\} - f$.
\end{cor}

\begin{prop}\label{breakprop}
Let $c_q$, $c_s$ and $c_{q+s}$ denote the conductors over $K$ of $K[q], K[s]$ and $K[q+s]$, respectively, and let $c_r$ denote the conductor of $K[q,r]$ over $K[q]$. Then the lower ramification breaks of $L$ over $K$ are $\ell_1 = \min\{c_{q+s},c_q\}$, $\ell_2 = 2c_s - \min\{c_{q+s},c_q\}$ and $\ell_3 = 2c_q + 2c_r - 2c_s - \min\{c_{q+s},c_q\}$, and the upper ramification breaks of $L$ over $K$ are $u_1 = \min\{c_{q+s},c_q\}, u_2 = c_s$ and $u_3 = (c_q + c_r)/2$.
\end{prop}

\begin{proof}
Let $c_L$ denote the conductor of $L$ over $K[q,s]$, let $C_q$ denote the conductor of $K[q,s]$ over $K[q]$, and let $\Gamma = \text{Gal}(L|K)$. By Propostion~\ref{d4struc}, $K[q,s]$ is the unique normal degree four subfield of $L$; thus, $\text{Gal}(L|K[q,s])$ is the only normal subgroup of $\Gamma$ of order two. By  Proposition IV.1 in ~\cite{LocalFields}, $\Gamma_i$ is a normal subgroup of $\Gamma$ for all $i$. In light of Proposition IV.2 in ~\cite{LocalFields}, it follows that $\ell_3 = c_L$. Similarly, by Proposition IV.14 in ~\cite{LocalFields}, $u_1$ and $u_2$ equal the first and second upper ramification breaks of $K[q,s]$ over $K$, respectively.

Recall that $f \le g$ by hypothesis; thus $c_q = f \le g = c_s$ by Lemma~\ref{condlem}.

Suppose $c_q < c_s$. Then $c_{q+s} = c_s$, and, as in Lemma~\ref{raylem1}, the first and second upper ramification breaks of $K[q,s]$ over $K$ are $c_q$ and $c_s$, respectively. Hence $u_1 = c_q = \min\{c_{q+s}, c_q\}$ and $u_2 = c_s$. Moreover, $C_q = 2c_s - c_q = 2c_s - 2c_q + \min\{c_{q+s}, c_q\}$ by Lemma~\ref{raylem1}.

Now suppose $c_q = c_s$. Then $c_{q+s} \le c_q$, and
\[C_q = c_{q+s} = \min\{c_{q+s}, c_q\} = 2c_s - 2c_q + \min\{c_{q+s}, c_q\},\]
the inequality and first equality holding by Lemma~\ref{raylem2}. If $c_{q+s} < c_q$, then $u_1 = c_{q+s} = \min\{c_{q+s}, c_q\}$ and $u_2 = c_s$ as in Lemma~\ref{raylem1}. If $c_{q+s} = c_q$, then $(\text{Gal}(K[q,s]|K))^i = \text{Gal}(K[q,s]|K)$ for all $i \le c_q$ and
\[(\text{Gal}(K[q,s]|K))^i \subseteq \text{Gal}(K[q,s]|K[q]) \cap \text{Gal}(K[q,s]|K[s]) = 1\] for all $i > c_q$. Hence $u_1 = u_2 = \min\{c_{q+s}, c_q\} =  c_s$.

Since $u_1 = \min\{c_{q+s}, c_q\}$, and $u_2 = c_s$, it follows by Herbrand's Formula that $\ell_1 = u_1 = \min\{c_{q+s}, c_q\}$ and that $\ell_2 = 2u_2 - u_1 = 2c_s - \min\{c_{q+s}, c_q\}$.
To compute $\ell_3$ (and $u_3$), note that, in either of the cases above, 
\[C_q = 2c_s - 2c_q + \min\{c_{q+s}, c_q\} \le 2c_s - c_q.\] Moreover, by Lemma~\ref{condlem}, $c_r \ge 2g+f = 2c_s + c_q > C_q$. Thus
\[\ell_3 = c_L = 2c_r - C_q = 2c_q + 2c_r - 2c_s - \min\{c_{q+s}, c_q\}\]
by Lemma~\ref{raylem1}. Hence
\begin{align*}
u_3 &= \ell_1 + (\ell_2 - \ell_1)/2 + (\ell_3 - \ell_2)/4 = \ell_3/4 + \ell_2/4 + \ell_1/2 \\ &= \left(2c_q + 2c_r - 2c_s - \min\{c_{q+s}, c_q\}\right)/4 \\ & + \left(2c_s - \min\{c_{q+s}, c_q\}\right)/4 + \min\{c_{q+s}, c_q\}/2 \\ &= (c_q + c_r)/2,
\end{align*}
the first equality holding by Herbrand's Formula.  
\end{proof}

Applying Lemma~\ref{condlem} to Proposition~\ref{breakprop} yields the following corollary.

\begin{cor}\label{breakcor} The lower ramification breaks of $L$ over $K$ are $\ell_1 = \min\{d,f\}$, $\ell_2 = 2g - \min\{d,f\}$ and $\ell_3 = 4\max\{f+g,h\} - 2g - \min\{d,f\}$, and the upper ramification breaks of $L$ over $K$ are $u_1 = \min\{d,f\}, u_2 = g$ and $u_3 = \max\{f+g,h\}$.
\end{cor}

\subsection{Characterization of Sequences of Ramification Breaks}\label{charsubsec}

In this subsection, we once again suppose that $k$ is algebraically closed. By Corollary~\ref{stdformcor} and Remark~\ref{totramrmk}, respectively, it follows that every $D_4$-extension of $K$ is both a $D_4$-standard form extension of $K$ and a totally ramified extension of $K$. Moreover, by Proposition~\ref{triplesprop}, every $D_4$-extension of $K$ has a `standard form' triple $(F',G',H')$ satisfying the additional condition $\deg_{t^{-1}} F' \le \deg_{t^{-1}} G'$.

We define the sequence of ramification groups of $L$ over $K$ to be a \emph{Type I} sequence if the sequence's second element is isomorphic to
$\mathbb{Z}/2\mathbb{Z} \times \mathbb{Z}/2\mathbb{Z}$, to be a \emph{Type II} sequence if the sequence's second element is isomorphic to
$\mathbb{Z}/4\mathbb{Z}$, and as a \emph{Type III} sequence if the sequence's second element is isomorphic to
$\mathbb{Z}/2\mathbb{Z}$. Note that in all cases, the second ramification break is strictly smaller than the third; thus the sequence's third element is always isomorphic to $\mathbb{Z}/2\mathbb{Z}$.

The type of the sequence of ramification groups of the extension $L$ over $K$ informs, to a large extent, which of the equicharacteristic deformations in Section~\ref{defsec} may and will be applied to the extension $L$. Moreover, the type of an extension's sequence of ramification groups affects the possible sequences of lower and of upper ramification breaks of that extension significantly. In this subsection, we consider (in the case where $k$ is algebraically closed) the relation between the type of an extension's sequence of ramification groups and the sequences of lower and of upper ramification breaks of that sequence exhaustively.


\begin{prop}\label{upperprop}
Let $(\alpha, \beta, \gamma) \in (\mathbb{Z}^+)^3$. Then $(\alpha, \beta, \gamma)$ is the sequence of upper ramification breaks for a $D_4$-extension of $K$ if and only if $\alpha$ is odd, $\alpha \le \beta$, $\beta$ is odd, $\gamma \ge \alpha + \beta$, and $\gamma$ is odd if $\gamma \notin \{\alpha + \beta, 2\beta\}$. Moreover, if $M$ is a $D_4$-extension of $K$ with sequence of upper ramification breaks $(\alpha, \beta, \gamma)$, then 
\begin{enumerate}[label=\emph{(\arabic*)}]
\item $M$ has a Type I sequence of ramification groups if $\gamma < 2\beta$.
\item $M$ has a Type II sequence of ramification groups if $\alpha < \beta$ and $\gamma = 2\beta$.
\item $M$ has a Type I or a Type II sequence of ramification groups if $\alpha < \beta$ and $\gamma > 2\beta$.
\item $M$ has a Type III sequence of ramification groups if and only if $\alpha = \beta$.
\end{enumerate}
\end{prop}

\begin{proof}
Corollary~\ref{triplescor} specifies a surjection $\Phi$ from the set $\mathcal{D}$ of triples $(\phi, \gamma, \eta)$ of standard form elements of $K$ such that $\phi, \gamma$ and $0$ are pairwise distinct to the set of $D_4$-extensions of $K$. Thus $(\alpha, \beta, \gamma)$ is the sequence of upper ramification breaks for a $D_4$-extension of $K$ if and only if there is a triple in $\mathcal{D}$ whose image under $\Phi$ has $(\alpha, \beta, \gamma)$ as its sequence of upper ramification breaks.

Let $(F, G, H) \in \mathcal{D}$, let $f = \deg_{t^{-1}}(F)$, $g = \deg_{t^{-1}}(G)$, $h = \deg_{t^{-1}}(H)$, and $d = \deg_{t^{-1}}(F+G)$. By Proposition~\ref{triplesprop}, we may and do assume, without loss of generality, that $f\le g$. Then the sequence of upper ramification breaks of $L = \Phi((F,G,H))$ is $(u_1, u_2, u_3) = (\min\{d,f\}, g, \max\{f+g, h\})$ by Corollary~\ref{breakcor}.
Moreover, the only restrictions on $f$, $g$, $h$ and $d$ 
general to all such triples
are that $f$, $d$ and $g$ must all be both odd and positive, that $h$ must be either both odd and positive or equal to zero, that $d = g$ if $f < g$, and that $d \le f$ if $f = g$. From these restrictions it follows that $u_1$ must be odd, that $u_1 \le u_2$, that $u_2$ must be odd, and that $u_3 \ge u_1 + u_2$.

Suppose firstly that $f < g = d$. Then $u_1 = f < g = u_2$. Hence the second element of the sequence of ramification groups of $L$ over $K$ is $\text{Gal}(L|K[q]) \cong \mathbb{Z}/2\mathbb{Z} \times \mathbb{Z}/2\mathbb{Z}$; \emph{i.e.}, $L$ has a Type I sequence of ramification groups. Moreover, $u_3 = \max\{u_1 + u_2, h\}$. Thus the additional restrictions on $(u_1, u_2, u_3)$ in this case are precisely that $u_1 < u_2$ and that $u_3$ must be odd if $u_3 \neq u_1 + u_2$.

Suppose secondly that $d < f = g$. Then $u_1 = d < g = u_2$. Hence the second element of the sequence of ramification groups of $L$ over $K$ is $\text{Gal}(L|K[q+s]) \cong \mathbb{Z}/4\mathbb{Z}$; \emph{i.e.}, $L$ has a Type II sequence of ramification groups. Moreover, $u_3 = \max\{2u_2, h\}$. Thus the additional restrictions on $(u_1, u_2, u_3)$ in this case are precisely that $u_1 < u_2$ , that $u_3 \ge 2u_2$, and that $u_3$ must be odd if $u_3 \neq 2u_2$.

Suppose thirdly that $d = f = g$. Then $u_1 = g = u_2$. Hence the second element of the sequence of ramification groups of $L$ over $K$ is $\text{Gal}(L|K[q,s]) \cong \mathbb{Z}/2\mathbb{Z}$; \emph{i.e.}, $L$ has a Type III sequence of ramification groups. Moreover, $u_3 = \max\{2u_2, h\}$. Thus the additional restrictions on $(u_1, u_2, u_3)$ in this case are precisely that $u_1 = u_2$, that $u_3 \ge 2u_2$, and that $u_3$ must be odd if $u_3 \neq 2u_2$.

Note that in all cases, $u_3$ must be odd if $u_3 \notin \{u_1 + u_2, 2u_2\}$, and that this is the only additional general restriction to those listed above. The unnumbered claim of the proposition now follows. Moreover, statement (4) holds since $u_1 < u_2$ in the first and second cases and $u_1 = u_2$ in the third case. Since $u_3 \ge 2u_2$ in the second case and $2u_2 > u_3 \ge u_1 + u_2$ implies that $u_1 < u_2$, statement (1) holds as well. Finally, statements (2) and (3) both hold since $u_3$ must be odd in the first case if $u_3 > u_1 + u_2$, and since there is no restriction in either the first or the second case on $u_3$ if $u_3 > 2u_2 > u_1 + u_2$, save that in both cases $u_3$ must be odd.
\end{proof}

The following proposition is the precise analogue to Proposition~\ref{upperprop} concerning the lower ramification breaks of $D_4$; accordingly, we omit its proof.

\begin{prop}
Let $(a, b, c) \in (\mathbb{Z}^+)^3$. Then $(a,b,c)$ is the sequence of lower ramification breaks for a $D_4$-extension of $K$ if and only if $a$ is odd, $a \le b$, $a \equiv b \pmod 4$, $c \ge 4a + b$, and $b \equiv c \pmod 8$ if $c \notin \{4a + b, 2a + 3b\}$. Moreover, if $M$ is a $D_4$-extension of $K$ with sequence of lower ramification breaks $(a, b, c)$, then 
\begin{enumerate}[label=\emph{(\arabic*)}]
\item $M$ has a Type I sequence of ramification groups if $c < 2a + 3b$.
\item $M$ has a Type II sequence of ramification groups if $a < b$ and $c = 2a + 3b$.
\item $M$ has a Type I or a Type II sequence of ramification groups if $a < b$ and $c > 2a + 3b$.
\item $M$ has a Type III sequence of ramification groups if and only if $a = b$.
\end{enumerate}
\end{prop}


\section{Deformations in Characteristic Two}\label{defsec}

We are now ready to define the equicharacteristic deformations needed to prove that $D_4$ is indeed a local Oort group. Let $k$ be an algebraically closed field of characteristic $p > 0$, let $K = k((t))$ be the field of Laurent series over $k$, fix an algebraic closure $K^{\text{alg}}$ of $K$, and let $L\subseteq K^{\text{alg}}$ be a Galois extension of $K$ with cyclic-by-$p$ Galois group $\Gamma$. Furthermore, let $A = k[[t]]$, and let $B \cong k[[z]]$ be the integral closure of $A$ in $L$. Finally, let $\mathcal{A} = k[[\varpi, t]]$, let $\mathcal{K} = \text{Frac}(\mathcal{A})$, and let $\mathcal{S} = \mathcal{A}[\varpi^{-1}]$, where $\varpi$ is an element transcendental over $A$.

\begin{defn}\label{equichardef}
An \emph{equicharacteristic deformation} of the $\Gamma$-extension $B$ over $A$ is a $\Gamma$-extension $k[[\varpi, z]]$ over $\mathcal{A}$ such that the Galois action of $\Gamma$ on $k[[\varpi, z]]$ over $\mathcal{A}$ restricts to the Galois action of $\Gamma$ on the extension $B$ over $A$.
\end{defn}

\begin{remark}
The original extension $B$ over $A$ is the special fiber of the deformation, while the extension $k[[\varpi, z]][\varpi^{-1}]$ over $\mathcal{S}$ is the generic fiber.  One can think of $\mathcal{S}$ as the ring of functions on the open unit disc of $k((\varpi))$ about $t$.
\end{remark}

Since we shall only be concerned with the case in which $p=2$ and $\Gamma \cong D_4$, we assume that $p=2$ and that $\Gamma \cong D_4$ henceforth. We shall define the needed equicharacteristic deformations by deforming, in a few particular ways, a triple of standard form elements that generates the $D_4$-extension $L$ of $K$. Accordingly, let $F$, $G$ and $H$ be elements of $K = k((t))$ in standard form with respect to $t$, and let $q$, $r$, and $s$ be elements of $K^{\text{alg}}$ such that, firstly, $q^2 + q = F$, $r^2 + r = Gq + H$ and $s^2 + s = G$, and, secondly, $L$ is the Galois closure over $K$ of $K[q,r]$. The existence of these elements in guaranteed by Proposition~\ref{existprop}. Let $f$, $g$, $h$ and $d$ denote the degrees in $t^{-1}$ of $F$, $G$, $H$ and $F+G$, respectively. Moreover, for $1 \le i \le 3$, let $u_i$ denote the $i$th upper ramification break of $L$ over $K$, and let $\ell_i$ denote the $i$th lower ramification break of $L$ over $K$.


\subsection{Preparatory Lemmas}\label{subprep}

Let $\widetilde{F}, \widetilde{G}$, and $\widetilde{H} \in \mathcal{K}$, and let $\tilde{q}, \tilde{r}, \tilde{s} \in \mathcal{K}^{\text{alg}}$ such that $\tilde{q}^2 + \tilde{q} = \widetilde{F}$, $\tilde{r}^2 + \tilde{r} = \widetilde{G}\tilde{q} + \widetilde{H}$. Also, let $\varrho \in k[[\varpi]]$. Then $\widehat{\mathcal{S}}_{(t-\varrho)} \cong k((\varpi))[[t-\varrho]]$, and 
$\widehat{\mathcal{K}}_{(t-\varrho)} = \text{Frac}(\widehat{\mathcal{S}}_{(t-\varrho)}) \cong k((\varpi))((t-\varrho))$.

\begin{lem}\label{finextlem1}
There exist a finite extension $k((\alpha))\subseteq k((\varpi))^{\emph{alg}}$ of $k((\varpi))$ and elements $F', G' \in k((\alpha))((t-\varrho))$ in standard form with respect to $t-\varrho$ such that $[\widetilde{F}] = [F']$ and $[\widetilde{G}] = [G']$ over $k((\alpha))((t-\varrho))$.
\end{lem}

\begin{proof}
Let $\widetilde{F} = \sum_{n \ge -N} \phi_n (t-\varrho)^n$, and let $\widetilde{G} = \sum_{n \ge -N} \gamma_n (t-\varrho)^n$, where each coefficient $\phi_n$ and each coefficient $\gamma_n$ is in $k((\varpi))$. Define $\alpha \in k((\varpi))^{\text{alg}}$ such that $k((\alpha))((t-\varrho))$ is the finite extension of $k((\varpi))((t-\varrho))$ given by appending Artin--Schreier roots of $\phi_0$ and $\gamma_0$, and, for all $d = 2^\ell m$, $m$ being odd, the $2^{\ell}$-th root of $\phi_{-d}$ and of $\gamma_{-d}$. Then $[\phi_0] = [\gamma_0] = 0$ over $k((\alpha))((t-\varrho))$, and, for all $d = 2^\ell m$, $m$ being odd,
\[ [\phi_{-d}(t-\varrho)^{-d}] = [\phi_{-d}^{2^{-\ell}}(t-\varrho)^{-m}], \quad\text{ and }\quad [\gamma_{-d}(t-\varrho)^{-d}] = [\gamma_{-d}^{2^{-\ell}}(t-\varrho)^{-m}] \]

Hence, as in the proof of Proposition~\ref{stdformprop}, each of $\widetilde{F}$ and $\widetilde{G}$ is Artin--Schreier-equivalent over $k((\alpha))((t-\varrho))$ to an element in standard form with respect to $t-\varrho$.
\end{proof}


Let $q' \in k((\alpha))((t-\varrho))^{\text{alg}}$ such that $(q')^2 + (q') = F'$, let $s' \in k((\alpha))((t-\varrho))^{\text{alg}}$ such that $(s')^2 + (s') = G'$, and let $\widetilde{J} = G'(q' + \tilde{q}) +  \widetilde{F}(s' + \tilde{s})^2 + \widetilde{H}$.

\begin{lem}\label{finextlem2}
There exist a finite extension $k((\alpha'))\subseteq k((\varpi))^{\emph{alg}}$ of $k((\alpha))$ and $J \in k((\alpha'))((t-\varrho))$ in standard form with respect to $t-\varrho$ 
such that $[\widetilde{J}] = [J]$ over $k((\alpha'))((t-\varrho))$.
\end{lem}

\begin{proof}
The proof of this lemma is entirely analogous to that of Lemma~\ref{finextlem1}.
\end{proof}

\begin{lem}\label{totramlem}
There exists a finite extension $k((\beta))\subseteq k((\varpi))^{\emph{alg}}$ of $k((\alpha'))$ such that each degree two extension $K_2|K_1$ of fields satisfying
\[\widehat{\mathcal{L}}_{(t-\varrho)} \supseteq K_2 \supseteq K_1 \supseteq \widehat{\mathcal{K}'}_{(t-\varrho)} \cong k((\beta))((t-\varrho)),\]
where $\mathcal{K}'$ denotes the fraction field of $k[[\beta, t]]$, and $\mathcal{L}$ denotes the Galois closure of $\mathcal{K}'[\tilde{q},\tilde{r}]$, is totally ramified.
\end{lem}

\begin{proof}
By appending elements to $k((\alpha'))$ as in Lemma~\ref{finextlem1}, we generate a finite extension $k((\beta))$ of $k((\alpha'))$ such that each degree two extension $K_2|K_1$ of fields satisfying $\widehat{\mathcal{L}}_{(t-\varrho)} \supseteq K_2 \supseteq K_1 \supseteq \widehat{\mathcal{K}'}_{(t-\varrho)}$ is generated by an Artin--Schreier root of an element in $K_1$ with odd valuation. Proposition~\ref{ramprop} then implies that each such extension is a totally ramified extension of fields.
\end{proof}

Now let $\mathcal{A}' = k[[\beta, t]]$, let $\mathcal{K}' = \text{Frac}(\mathcal{A}')$ (as in Lemma~\ref{totramlem}), and let $\mathcal{S}' = \mathcal{A}'[\beta^{-1}]$. Moreover, let $\mathcal{L}$ be the Galois closure of $\mathcal{K}'[\tilde{q}, \tilde{r}]$ (as in Lemma~\ref{totramlem}), let $\mathcal{B}$ be the integral closure of $\mathcal{A}'$ in $\mathcal{L}$, and let $\mathcal{T} = \mathcal{B}[\varpi^{-1}]$. 

\begin{cor}\label{stdcor}
Each of the factors of the degree eight $\widehat{\mathcal{K}'}_{(t-\varrho)}$-algebra $\widehat{\mathcal{L}}_{(t-\varrho)}$ is both totally ramified over and generated by standard form elements over $\widehat{\mathcal{K}'}_{(t-\varrho)} \cong k((\beta))((t-\varrho))$.
\end{cor}

\begin{proof}
The first claim of the corollary follows immediately from Lemma~\ref{totramlem}. For the second claim, let $r' \in k((\beta))((t-\varrho))$ such that $(r')^2 + r' = G'q' + J$. By Proposition~\ref{nongalprop}, $\mathcal{K}'[\tilde{q}, \tilde{r}] = \mathcal{K}'[q', r']$. The second claim of the corollary now follows.
\end{proof}


\begin{lem}\label{d4lem}
Suppose that $\widetilde{F}, \widetilde{G} \in \mathcal{A}'_{(\beta)} \cap \mathcal{K} = \mathcal{A}_{(\varpi)}$, that $\widetilde{F} \equiv F \pmod{\varpi}$ and that $\widetilde{G} \equiv G \pmod{\varpi}$.
Then $\emph{Gal}(\mathcal{L}|\mathcal{K}') \cong D_4$.
\end{lem}

\begin{proof}
Since $\text{Gal}(\mathcal{L}|\mathcal{K}') \cong D_4$, it follows that
$[F] \neq 0$ over $K$, that $[G] \neq 0$ over $K$, and that $[G] \neq [F]$ over $K$ by Proposition~\ref{nongalprop}.
Since $\widetilde{F} \equiv F \pmod{\varpi}$ and $\widetilde{G} \equiv G \pmod{\varpi}$, it follows that $\widetilde{F} \equiv F \pmod{\beta}$ and $\widetilde{G} \equiv G \pmod{\beta}$.  Hence $[\widetilde{F}] \neq 0$ over $\mathcal{A}'_{(\beta)}$, $[\widetilde{G}] \neq 0$ over $\mathcal{A}'_{(\beta)}$ and $[\widetilde{G}] \neq [\widetilde{F}]$ over $\mathcal{A}'_{(\beta)}$. 
Moreover, Since $\mathcal{A}'_{(\beta)}$ is a discrete valuation ring (and hence is integrally closed), it follows that $[\widetilde{F}] \neq 0$ over $\mathcal{K}'$, $[\widetilde{G}] \neq 0$ over $\mathcal{K}'$ and $[\widetilde{G}] \neq [\widetilde{F}]$ over $\mathcal{K}'$. Therefore, $\text{Gal}(\mathcal{L}|\mathcal{K'}) \cong D_4$ by Proposition~\ref{classprop} and by Lemma~\ref{kerlem1}.
\end{proof}


\subsection{First Deformation}

For the first equicharacteristic deformation, suppose that the sequence of ramification groups of $L$ over $K$ is of Type I, \emph{i.e.}, that $f < d = g$. 
Let $\widetilde{F} = F$, $\widetilde{G} = Gt^2(t-\varpi)^{-2}$, and
$\widetilde{H} = Ht^2(t-\varpi)^{-2}$. By Corollary~\ref{stdcor}, there exists a finite extension $k((\beta))$ of $k((\varpi))$ such that 
each of the factors of the degree eight $\widehat{\mathcal{K}'}_{(t-\varpi)}$-algebra $\widehat{\mathcal{L}}_{(t-\varpi)}$ is both totally ramified over and generated by standard form elements over $\widehat{\mathcal{K}'}_{(t-\varpi)} \cong k((\beta))((t-\varpi))$,
where $\mathcal{A}'$, $\mathcal{K}'$, $\mathcal{S}'$, $\mathcal{L}$, $\mathcal{B}$ and $\mathcal{T}$ are defined as in Subsection~\ref{subprep}.

\begin{prop}[First Deformation]\label{kleindef}


The following statements all hold.

\begin{enumerate}[label=\emph{(\arabic*)}]
\item $\emph{Gal}(\mathcal{L}|\mathcal{K}') \cong D_4$.
\item As a $D_4$-extension of Dedekind domains, $\mathcal{B}/(\beta)$ over $\mathcal{A}'/(\beta)$ is isomorphic to $B$ over $A$.
\item The $D_4$-extension of Dedekind domains $\mathcal{T}$ over $\mathcal{S}'$ is branched at precisely two maximal ideals, \emph{viz}.~$(t)$ and $(t-\varpi)$.
Above $(t)$, the inertia group is $D_4$, the sequence of lower ramification breaks is $(\ell_1, \ell_2 - 4, \ell_3 - 4)$, and the sequence of upper ramification breaks is $(u_1, u_2 - 2, u_3 - 2)$.
Above $(t-\varpi)$, the inertia group is $\emph{Gal}(\mathcal{L}|\mathcal{K}'[\tilde{q}]) \cong \mathbb{Z}/2\mathbb{Z} \times \mathbb{Z}/2\mathbb{Z}$, and the sequence of lower ramification breaks is $(1,1)$.
\end{enumerate}
\end{prop}

\begin{proof}


To prove (1), note that $\widetilde{F}, \widetilde{G} \in \mathcal{A}_{(\varpi)} = \mathcal{A}'_{(\beta)} \cap \mathcal{K}$ since $(t-\varpi)^{-2} = t^{-2}\sum_{n=0}^\infty (t^{-1}\varpi)^{2n}$. Note also that $\widetilde{F} \equiv F \pmod{\varpi}$, and that $\widetilde{G} \equiv G \pmod{\varpi}$. Thus (1) holds by Lemma~\ref{d4lem}. 


To prove (3) for the ideal $(t)$, consider the completion $\widehat{\mathcal{S}'}_{(t)} \cong k((\beta))[[t]]$ of the localization $\mathcal{S}'_{(t)}$ of $\mathcal{S}'$, and note that, over $k((\beta))[[t]]$, $(t - \varpi)^{-2}$ is a unit. Thus $-v_{(t)}(\widetilde{F}) = f$, $-v_{(t)}(\widetilde{G}) = g-2$
and $-v_{(t)}(\widetilde{H}) = h-2$ (unless $H = 0$).
Since $f$ is odd, $\widehat{\mathcal{K}'}_{(t)}[\tilde{q}]$ is a totally ramified extension of $\widehat{\mathcal{K}'}_{(t)} \cong k((\beta))((t))$ with conductor $f$ by Proposition~\ref{ramprop}. Similarly, $\widehat{\mathcal{K}'}_{(t)}[\tilde{s}]$ is a totally ramified extensions of $\widehat{\mathcal{K}'}_{(t)}$ with conductor $g-2$. Moreover, if $f < g-2$, $-v_{(t)}(\widetilde{F} + \widetilde{G}) = g-2$. Since $(t-\varpi)^{-2} = \varpi^{-2}\sum_{n=0}^\infty(t\varpi^{-1})^{2n}$, the coefficient of $t^{-g+2}$ in the Laurent series expansion of $\widetilde{G}$ is not contained in $k$, whereas the coefficient of $t^{-f}$ in the Laurent series expansion of $\widetilde{F}$ is contained in $k$. Thus, if $f = g-2$, $-v_{(t)}(\widetilde{F} + \widetilde{G}) = g-2$ as well. Therefore, by Propostion~\ref{ramprop}, $\widehat{\mathcal{K}'}_{(t)}[\tilde{q}+\tilde{s}]$ is ramified over $\widehat{\mathcal{K}'}_{(t)}$ with conductor $g-2$. Hence, by Corollary~\ref{breakcor}, the first, second and third terms in the sequence of upper ramification breaks over $(t)$ are $\min\{g-2, f\} = f = u_1$, $g-2 = u_2 - 2$ and $\min\{f + g-2, h-2\} = \max\{f + g, h\} - 2 = u_3 - 2$, respectively. 
Statement (3) for $(t)$ now follows by Herbrand's Formula.

To prove (3) for the ideal $(t-\varpi)$, note that, over the completion $\widehat{\mathcal{S}'}_{(t-\varpi)} \cong k((\beta))[[t-\varpi]]$ of the localization $\mathcal{S}'_{(t-\varpi)}$ of $\mathcal{S}'$, $t$ is a unit. Thus $-v_{(t-\varpi)}(\widetilde{F}) = 0$, $-v_{(t-\varpi)}(\widetilde{G}) = -v_{(t-\varpi)}(\widetilde{F} + \widetilde{G}) = 2$, and $-v_{(t-\varpi)}(\widetilde{H}) \le 2$. 
Since each factor of $\widehat{\mathcal{L}}_{(t-\varpi)}$ is generated by standard form elements over $\widehat{\mathcal{K}'}_{(t-\varpi)}$, it follows that $[\widetilde{F}] = 0$ over $\widehat{\mathcal{K}'}_{(t-\varpi)}$ and that thus $\widehat{\mathcal{K}'}_{(t-\varpi)}[\tilde{q}] = \widehat{\mathcal{K}'}_{(t-\varpi)}$.
Furthermore, the conductor of $\widehat{\mathcal{K}'}_{(t-\varpi)}[\tilde{q},\tilde{s}]$ over $\widehat{\mathcal{K}}_{(t-\varpi)}[\tilde{q}]$ is 1.


Since $[\widetilde{F}] = 0$ over $\widehat{\mathcal{K}'}_{(t-\varpi)}$, the fact that each factor of $\widehat{\mathcal{L}}_{(t-\varpi)}$ is generated by standard form elements over $\widehat{\mathcal{K}'}_{(t-\varpi)}$ implies that
$\widetilde{G}\tilde{q} + \widetilde{H}$ is Artin--Schreier-equivalent over $\widehat{\mathcal{K}'}_{(t-\varpi)}$ to an element $J$ in standard form with respect to $t-\varpi$.
Moreover, since $\tilde{q} \notin k((\beta))$, $-v_{(t-\varpi)}(\widetilde{G}\tilde{q} + \widetilde{H}) = 2$. Thus $-v_{(t-\varpi)}(J) = 1$, and the conductor of $\widehat{\mathcal{K}'}_{(t-\varpi)}[\tilde{q}, \tilde{r}]$ over $\widehat{\mathcal{K}'}_{(t-\varpi)}[\tilde{q}]$ is 1. Similarly,  $-v_{(t-\varpi)}(\widetilde{G}\tilde{q} + \widetilde{G} + \widetilde{H}) = 2$, and the conductor of $\widehat{\mathcal{K}'}_{(t-\varpi)}[\tilde{q}, \tilde{r}+\tilde{s}]$ over $\widehat{\mathcal{K}'}_{(t-\varpi)}[\tilde{q}]$ is 1. Statement (3) for $(t-\varpi)$ now follows by Lemma~\ref{raylem2}.

Finally, to prove (2), note that the degree $\delta_{L|K}$ of the different of $L$ over $K$ is $4\ell_1 + 2\ell_2 + \ell_3 + 7$ by Hilbert's different formula ~\cite{LocalFields}. Similarly, by (3), the contribution of $(t)$ to the degree $\delta_{\mathcal{T}'|\mathcal{S}'}$ of the different of $\mathcal{T}'$ over $\mathcal{S}'$ is $4\ell_1 + 2(\ell_2 - 4) + (\ell_3 - 4) + 7 = \delta_{L|K} - 12$, and the contribution of $(t-\varpi)$ is $\delta_{\mathcal{T}'|\mathcal{S}'}$ is $2\cdot(2(1) + 1 + 3) = 12$
. Thus $\delta_{\mathcal{T}'|\mathcal{S}'} = \delta_{L|K} - 12 + 12 = \delta_{L|K}$. Therefore (2) holds by Theorem 3.4 in ~\cite{GM98}.
\end{proof} 

\subsection{Second Deformation}
For the second equicharacteristic deformation, suppose that the sequence of ramification groups of $L$ over $K$ is of Type II, \emph{i.e.}, that $d < f = g$. 
Let $a_f$ denote the coefficient of $t^{-f}$ in the Laurent series expansion of $F$, and let $a_g$ denote the coefficient of $t^{-g}$ in the Laurent series expansion of $G$. Let also $\widetilde{F} = F + a_ft^{-f} + a_ft^{-f + 2}(t-\varpi)^{-2}$, $\widetilde{G} = G + a_gt^{-g} + a_gt^{-g + 2}(t-\varpi)^{-2}$, and
$\widetilde{H} = Ht^4(t-\varpi)^{-4}$. By Corollary~\ref{stdcor}, there exists a finite extension $k((\beta))$ of $k((\varpi))$ such that 
each of the factors of the degree eight $\widehat{\mathcal{K}'}_{(t-\varpi)}$-algebra $\widehat{\mathcal{L}}_{(t-\varpi)}$ is both totally ramified over and generated by standard form elements over $\widehat{\mathcal{K}'}_{(t-\varpi)} \cong k((\beta))((t-\varpi))$,
where $\mathcal{A}'$, $\mathcal{K}'$, $\mathcal{S}'$, $\mathcal{L}$, $\mathcal{B}$ and $\mathcal{T}$ are defined as in Subsection~\ref{subprep}.

\begin{prop}[Second Deformation]\label{cyclicdef}
%

The following statements all hold.

\begin{enumerate}[label=\emph{(\arabic*)}]
\item $\emph{Gal}(\mathcal{L}|\mathcal{K}') \cong D_4$.
\item As a $D_4$-extension of Dedekind domains, $\mathcal{B}/(\beta)$ over $\mathcal{A}'/(\beta)$, is isomorphic to $B$ over $A$.
\item The $D_4$-extension of Dedekind domains $\mathcal{T}$ over $\mathcal{S}'$ is branched at precisely two maximal ideals, \emph{viz}.~$(t)$ and $(t-\varpi)$.
Above $(t)$, the inertia group is $D_4$, the sequence of lower ramification breaks is $(\ell_1, \ell_2 - 4, \ell_3 - 12)$, and the sequence of upper ramification breaks is $(u_1, u_2 - 2, u_3 - 4)$.
Above $(t-\varpi)$, the inertia group is $\emph{Gal}(\mathcal{L}|\mathcal{K}'[\tilde{q} + \tilde{s}]) \cong \mathbb{Z}/4\mathbb{Z}$, and the sequence of lower ramification breaks is $(1,5)$.
\end{enumerate}
\end{prop}

\begin{proof}


To prove (1), note that $\widetilde{F}, \widetilde{G} \in \mathcal{A}_{(\varpi)} = \mathcal{A}'_{(\beta)} \cap \mathcal{K}$ since $(t-\varpi)^{-2} = t^{-2}\sum_{n=0}^\infty (t^{-1}\varpi)^{2n}$. Note also that $\widetilde{F} \equiv F \pmod{\varpi}$, and that $\widetilde{G} \equiv G \pmod{\varpi}$. Thus (1) holds by Lemma~\ref{d4lem}. 


To prove (3) for the ideal $(t)$, note that, over the completion $\widehat{\mathcal{S}'}_{(t)} \cong k((\beta))[[t]]$ of the localization $\mathcal{S}'_{(t)}$ of $\mathcal{S}'$, $(t - \varpi)^{-2}$ is a unit. Thus $-v_{(t)}(\widetilde{F}) = f-2$ and $-v_{(t)}(\widetilde{G}) = g-2$, and $-v_{(t)}(\widetilde{H}) = h-4$ (unless $H = 0$).
Since $f-2$ is odd, $\widehat{\mathcal{K}'}_{(t)}[\tilde{q}]$ is a totally ramified extension of $\widehat{\mathcal{K}'}_{(t)} \cong k((\beta))((t))$ with conductor $f-2$ by Propostion~\ref{ramprop}. Similarly, $\widehat{\mathcal{K}'}_{(t)}[\tilde{s}]$ is totally ramified over $\widehat{\mathcal{K}'}_{(t)}$ with conductor $g-2$.
Moreover, since $d < f = g$, it follows that $a_f = a_g$ and that $\widetilde{F} + \widetilde{G} = F + G$. Thus $-v_{(t)}(\widetilde{F} + \widetilde{G}) = d$.
Therefore, since $f-2$, $g-2$ and $d$ are all both positive and odd, it follows by Corollary~\ref{breakcor} that $\widehat{\mathcal{L}}_{(t)}$ is totally ramified over $\widehat{\mathcal{K}'}_{(t)}$, and that the first, second and third terms of the sequence of upper ramification breaks over $(t)$ are $\min\{d, f-2\} = d = u_1$, $g - 2 = u_2 - 2$ and $\max\{f-2 + g-2, h-4\} = \max\{f + g, h\} - 4 = u_3 - 4$,  respectively. Statement (3) for $(t)$ now follows by Herbrand's Formula.

To prove (3) for the ideal $(t-\varpi)$, note that, over the completion $\widehat{\mathcal{S}'}_{(t-\varpi)} \cong k((\beta))[[t-\varpi]]$ of the localization $\mathcal{S}'_{(t-\varpi)}$ of $\mathcal{S}'$, $t$ is a unit. Thus $-v_{(t-\varpi)}(\widetilde{F}) = -v_{(t-\varpi)}(\widetilde{G}) = 2$, $-v_{(t-\varpi)}(\widetilde{F} + \widetilde{G}) = -v_{(t-\varpi)}(F + G) = 0$, and $-v_{(t-\varpi)}(\widetilde{H}) \le 4$.
Since each factor of $\widehat{\mathcal{L}}_{(t-\varpi)}$ is generated by standard form elements over $\widehat{\mathcal{K}'}_{(t-\varpi)}$, it follows that $[\widetilde{F} + \widetilde{G}] = 0$ over $\widehat{\mathcal{K}'}_{(t-\varpi)}$ and that thus $\widehat{\mathcal{K}'}_{(t-\varpi)}[\tilde{q}+\tilde{s}] = \widehat{\mathcal{K}'}_{(t-\varpi)}$.
Furthermore, the conductor of $\widehat{\mathcal{K}'}_{(t-\varpi)}[\tilde{q}] = \widehat{\mathcal{K}'}_{(t-\varpi)}[\tilde{s}]$ over $\widehat{\mathcal{K}'}_{(t-\varpi)}$ is 1.

Let $F'$ and $G'$ denote the elements of $\widehat{\mathcal{K}'}_{(t-\varpi)}$ in standard form that are Artin--Schreier-equivalent to $\widetilde{F}$ and $\widetilde{G}$, respectively, and let $q',s' \in \widehat{\mathcal{K}'}_{(t-\varpi)}^{\text{alg}}$ such that $(q')^2 + q' = F'$ and $(s')^2 + s' = G'$.
Let also $\widetilde{J} = G'(q' + \tilde{q}) + \widetilde{F}(s' + \tilde{s})^2 + \widetilde{H} \in \widehat{\mathcal{K}'}_{(t-\varpi)}$.
Then $[\widetilde{G}\tilde{q} + \widetilde{H}] = [G'q' + \widetilde{J}]$ by Proposition~\ref{nongalprop}.
Since $\widehat{\mathcal{L}}_{(t-\varpi)}$ is generated by standard form elements over $\widehat{\mathcal{K}'}_{(t-\varpi)}$, it follows that $\widetilde{J}$ is Artin--Schreier-equivalent over $\widehat{\mathcal{K}'}_{(t-\varpi)}$ to an element $J$ in standard form with respect to $t-\varpi$.

Let $b$ denote the conductor of $\widehat{\mathcal{K}'}_{(t-\varpi)}[\tilde{q},\tilde{r},\tilde{s}] = \widehat{\mathcal{K}'}_{(t-\varpi)}[\tilde{q}, \tilde{r}]$ over $\widehat{\mathcal{K}'}_{(t-\varpi)}[\tilde{q}]$.
It follows from Proposition IV.2 in~\cite{LocalFields} that the sequence of lower ramification breaks of $\widehat{\mathcal{K}'}_{(t-\varpi)}[\tilde{q}, \tilde{r}]$ over $\widehat{\mathcal{K}'}_{(t-\varpi)}[\tilde{q}]$ is $(1, b)$.
Since $-v_{(t-\varpi)}(F' + \widetilde{F}) = 2$, $-v_{(t-\varpi)}(q' + \tilde{q}) = 1$. Similarly, $-v_{(t-\varpi)}(s' + \tilde{s}) = 1$. Thus
\[-v_{(t-\varpi)}(\widetilde{J}) = G'(q' + \tilde{q}) + \widetilde{F}(s' + \tilde{s})^2 + \widetilde{H} \le 4,\]
and hence $-v_{(t-\varpi)}(J) \le 3$. Since $\widehat{\mathcal{K}'}_{(t-\varpi)}[\tilde{q}] = \widehat{\mathcal{K}'}_{(t-\varpi)}[\tilde{s}]$, Proposition~\ref{asprop} implies that $F' = G'$. Therefore, $b = 2\max\{1 + 1,-v_{(t-\varpi)}(J)\} - 1 \le 5$ by Corollary~\ref{condcor}.
Thus the contribution of $(t-\varpi)$ to the degree $\delta_{\mathcal{T}'|\mathcal{S}'}$ of the different of $\mathcal{T}'$ over $\mathcal{S}'$ is $2 \cdot (2(1) + b + 3) = 2b + 10\le 20$ by Hilbert's different formula. Moreover,
the contribution of $(t)$ to the degree $\delta_{\mathcal{T}'|\mathcal{S}'}$ is $4\ell_1 + 2(\ell_2 - 4) + (\ell_3 - 12) + 7 = \delta_{L|K} - 20$ by statement (3) for $(t)$. Hence $\delta_{\mathcal{T}'|\mathcal{S}'} \le \delta_{L|K}$. By Theorem 3.4 in ~\cite{GM98}, $\delta_{\mathcal{T}'|\mathcal{S}'} \ge \delta_{L|K}$. Thus $\delta_{\mathcal{T}'|\mathcal{S}'} = \delta_{L|K}$, $2b + 10 = 20$, and $b=5$. Statement (3) for $(t-\varpi)$ now follows immediately, and statement (2) follows by Theorem 3.4 in ~\cite{GM98}.
\end{proof} 

\subsection{Third Deformation}

For the third equicharacteristic deformation, suppose that $u_1 = \min\{d,f\} > 1$.
Let $\widetilde{F} = Ft^2(t-\varpi)^{-2}$, $\widetilde{G} = Gt^2(t-\varpi)^{-2}$, and $\widetilde{H} = Ht^4(t-\varpi)^{-4}$. By Corollary~\ref{stdcor}, there exists a finite extension $k((\beta))$ of $k((\varpi))$ such that 
each of the factors of the degree eight $\widehat{\mathcal{K}'}_{(t-\varpi)}$-algebra $\widehat{\mathcal{L}}_{(t-\varpi)}$ is both totally ramified over and generated by standard form elements over $\widehat{\mathcal{K}'}_{(t-\varpi)} \cong k((\beta))((t-\varpi))$,
where $\mathcal{A}'$, $\mathcal{K}'$, $\mathcal{S}'$, $\mathcal{L}$, $\mathcal{B}$ and $\mathcal{T}$ are defined as in Subsection~\ref{subprep}.

\begin{prop}[Third Deformation]\label{maindef}
The following statements all hold.

\begin{enumerate}[label=\emph{(\arabic*)}]
\item $\emph{Gal}(\mathcal{L}|\mathcal{K}') \cong D_4$.
\item As a $D_4$-extension of Dedekind domains, $\mathcal{B}/(\beta)$ over $\mathcal{A}'/(\beta)$ is isomorphic to $B$ over $A$.
\item The $D_4$-extension of Dedekind domains $\mathcal{T}$ over $\mathcal{S}'$ is branched at precisely two maximal ideals, \emph{viz}.~$(t)$ and $(t-\varpi)$.
Above $(t)$, the inertia group is $D_4$, the sequence of lower ramification breaks is $(\ell_1 - 2, \ell_2 - 2, \ell_3 - 10)$, and the sequence of upper ramification breaks is $(u_1 - 2, u_2 - 2, u_3 - 4)$.
Above $(t-\varpi)$, the inertia group is $D_4$, and the sequence of lower ramification breaks is $(1,1,9)$.
\end{enumerate}
\end{prop}

\begin{proof}


To prove (1), note that $\widetilde{F}, \widetilde{G} \in \mathcal{A}_{(\varpi)} = \mathcal{A}'_{(\beta)} \cap \mathcal{K}$ since $(t-\varpi)^{-2} = t^{-2}\sum_{n=0}^\infty (t^{-1}\varpi)^{2n}$. Note also that $\widetilde{F} \equiv F \pmod{\varpi}$, and that $\widetilde{G} \equiv G \pmod{\varpi}$. Thus (1) holds by Lemma~\ref{d4lem}. 


To prove (3) for the ideal $(t)$, note that, over the completion $\widehat{\mathcal{S}'}_{(t)} \cong k((\beta))[[t]]$ of the localization $\mathcal{S}'_{(t)}$ of $\mathcal{S}'$, $(t - \varpi)^{-2}$ is a unit. Thus $-v_{(t)}(\widetilde{F}) = f-2$, $-v_{(t)}(\widetilde{G}) = g-2$, and $-v_{(t)}(\widetilde{H}) = h-4$ (unless $H = 0$).
Since $f-2$ is both positive and odd, $\widehat{\mathcal{K}'}_{(t)}[\tilde{q}]$ is a totally ramified extension of $\widehat{\mathcal{K}'}_{(t)} \cong k((\beta))((t))$ with conductor $f-2$ by Propostion~\ref{ramprop}.
Similarly, $\widehat{\mathcal{K}'}_{(t)}[\tilde{s}]$ is totally ramified over $\widehat{\mathcal{K}'}_{(t)} \cong k((\beta))((t))$ with conductor $g-2$.
Moreover, since $\widetilde{F} + \widetilde{G} = (F+G)t^2(t-\varpi)^{-2}$,  it follows that $\widehat{\mathcal{K}'}_{(t)}[\tilde{q}+\tilde{s}]$ is totally ramified over $\widehat{\mathcal{K}'}_{(t)}$ with conductor $d-2$.
Since $f-2$, $g-2$ and $d-2$ are all both positive and odd, and $h-4$ is not both positive and even, Corollary~\ref{breakcor} implies that $\widehat{\mathcal{L}}_{(t)}$ is totally ramified over $\widehat{\mathcal{K}'}_{(t)}$, and that the first, second and third terms of the sequence of upper ramification breaks over $(t)$ are $\min\{d-2, f-2\} = \min\{d,f\} - 2 = u_1 - 2$, $g - 2 = u_2 - 2$ and $\max\{f-2 + g-2, h-4\} = \max\{f + g, h\} - 4 = u_3 - 4$,  respectively. Statement (3) for $(t)$ now follows by Herbrand's Formula.

To prove (3) for the ideal $(t-\varpi)$, note that, over the completion $\widehat{\mathcal{S}'}_{(t-\varpi)} \cong k((\beta))[[t-\varpi]]$ of the localization $\mathcal{S}'_{(t-\varpi)}$ of $\mathcal{S}'$, $t$ is a unit. Thus $-v_{(t-\varpi)}(\widetilde{F}) = 2 = -v_{(t-\varpi)}(\widetilde{G}) = -v_{(t-\varpi)}(\widetilde{F} + \widetilde{G}) = 2$, and $-v_{(t-\varpi)}(\widetilde{H}) \le 4$.
Since each factor of $\widehat{\mathcal{L}}_{(t-\varpi)}$ is generated by standard form elements over $\widehat{\mathcal{K}'}_{(t-\varpi)}$, it follows that each of the conductors of $\widehat{\mathcal{K}'}_{(t-\varpi)}[\tilde{q}]$, $\widehat{\mathcal{K}'}_{(t-\varpi)}[\tilde{s}]$, and $\widehat{\mathcal{K}'}_{(t-\varpi)}[\tilde{q} + \tilde{s}]$ over $\widehat{\mathcal{K}'}_{(t-\varpi)}$ is 1. Thus (\emph{cf.}~Remark~\ref{stdrmk}) $\widehat{\mathcal{L}}_{(t-\varpi)}$ is itself a (totally ramified) field extension of $\widehat{\mathcal{K}'}_{(t-\varpi)}$.

As in the second deformation, let $F'$ and $G'$ denote the elements of $\widehat{\mathcal{K}'}_{(t-\varpi)}$ in standard form that are Artin--Schreier-equivalent to $\widetilde{F}$ and $\widetilde{G}$, respectively, and let $q',s' \in \widehat{\mathcal{K}'}_{(t-\varpi)}^{\text{alg}}$ such that $(q')^2 + q' = F'$ and $(s')^2 + s' = G'$.
Let also $\widetilde{J} = G'(q' + \tilde{q}) + \widetilde{F}(s' + \tilde{s})^2 + \widetilde{H} \in \widehat{\mathcal{K}'}_{(t-\varpi)}$.
Then $[\widetilde{G}\tilde{q} + \widetilde{H}] = [G'q' + \widetilde{J}]$ by Proposition~\ref{nongalprop}.
Since $\widehat{\mathcal{L}}_{(t-\varpi)}$ is a $D_4$-standard form extension of $\widehat{\mathcal{K}'}_{(t-\varpi)}$, it follows that $\widetilde{J}$ is Artin--Schreier-equivalent over $\widehat{\mathcal{K}'}_{(t-\varpi)}$ to an element $J$ in standard form with respect to $t-\varpi$. 

Let $b$ denote the conductor of $\widehat{\mathcal{K}'}_{(t-\varpi)}[\tilde{q}, \tilde{r}]$ over $\widehat{\mathcal{K}'}_{(t-\varpi)}[\tilde{q}]$.
By Proposition~\ref{breakprop}, it follows that the sequence of lower ramification breaks of $\widehat{\mathcal{L}}_{(t-\varpi)}$ over $\widehat{\mathcal{K}'}_{(t-\varpi)}$ is $(1,1, 2b - 1)$.
Since $-v_{(t-\varpi)}(F' + \widetilde{F}) = 2$, $-v_{(t-\varpi)}(q' + \tilde{q}) = 1$. Similarly, $-v_{(t-\varpi)}(s' + \tilde{s}) = 1$. Thus
\[-v_{(t-\varpi)}(\widetilde{J}) = G'(q' + \tilde{q}) + \widetilde{F}(s' + \tilde{s})^2 + \widetilde{H} \le 4,\]
and hence $-v_{(t-\varpi)}(J) \le 3$. Therefore, $b = 2\max\{1 + 1, -v_{(t-\varpi)}(J)\} - 1 \le 5$ by Lemma~\ref{condlem}.
Thus the contribution of $(t-\varpi)$ to the degree $\delta_{\mathcal{T}'|\mathcal{S}'}$ of the different of $\mathcal{T}'$ over $\mathcal{S}'$ is $4(1) + 2(1) + (2b - 1) + 7 = 2b + 12 \le 22$ by Hilbert's different formula. Moreover, the contribution of $(t)$ to the degree $\delta_{\mathcal{T}'|\mathcal{S}'}$ is $4(\ell_1 - 2) + 2(\ell_2 - 2) + (\ell_3 - 10) + 7 = \delta_{L|K} - 22$ by statement (3) for $(t)$. Hence $\delta_{\mathcal{T}'|\mathcal{S}'} \le \delta_{L|K}$. By Theorem 3.4 in ~\cite{GM98}, $\delta_{\mathcal{T}'|\mathcal{S}'} \ge \delta_{L|K}$. Thus $\delta_{\mathcal{T}'|\mathcal{S}'} = \delta_{L|K}$, $2b + 12 = 22$ and $b=5$. Statement (3) for $(t-\varpi)$ now follows immediately, and statement (2) follows by Theorem 3.4 in ~\cite{GM98}.
\end{proof}

\section{Main Theorem}

Having now found various equicharacteristic deformations of $D_4$-Galois extensions of complete discrete valuation fields of characteristic two with algebraically closed residue field, we use the `method of equicharacteristic deformation', as used in~\cite{Pop14}, in~\cite{ObusD9}, and in \cite{ObusA4} to prove that all such extensions lift to characteristic zero, \emph{i.e.}, that $D_4$ is a local Oort group for the prime two.  We begin by using the deformations of Section~\ref{defsec} to reduce to the case of extensions with, in some sense, small ramification breaks.

\subsection{Deformation Reductions}

In order to use the deformations of Section~\ref{defsec} effectively to reduce the cases under consideration, we shall need to use Theorem 6.20 in~\cite{Obus17arXiv}, which is reproduced below as Theorem~\ref{mumfthm} for convenience. The argument for this theorem was communicated orally by Pop, who presented an earlier version of this theorem, peculiar to the cyclic case, in~\cite{Pop14}.

Let $k$ be an algebraically closed residue field of characteristic $p > 0$, let $K = k((t))$, and let $G$ be a cyclic-by-$p$ group.

\begin{thm}\label{mumfthm}
Suppose that $k[[z]]/k[[t]]$ is a local $G$-extension that admits an equicharacteristic deformation whose generic fiber lifts to characteristic zero after base change to the algebraic closure. Then $k[[z]]/k[[t]]$ lifts to characteristic zero.
\end{thm}

As we shall only require the case in which $p = 2$, we shall assume that $p = 2$ henceforth.

\begin{prop}\label{inductionprop}
Let $(u_1, u_2, u_3)$ be a triple of positive integers such that there exists a $D_4$-extension of $K$
whose sequence of ramification breaks $(u_1, u_2, u_3)$. 
Suppose that $u_2 > 1$, and that every $D_4$-Galois extension of $K$ with second ramification break over $K$ less than or equal to $u_2 - 2$ lifts to characteristic zero. Then every $D_4$-Galois extension of $K$ 
whose sequence of ramification breaks is $(u_1, u_2, u_3)$ lifts to characteristic zero.
\end{prop}

\begin{proof}
Let $L$ be a $D_4$-extension of $K$ whose sequence of upper ramification breaks is $(u_1, u_2, u_3)$. The sequence of ramification groups must be of one of the three types enumerated in Subection~\ref{charsubsec}.

Suppose firstly that the sequence of ramification groups of $L$ is of Type I. By Proposition~\ref{kleindef}, $L$ admits an equicharacteristic deformation whose generic fiber has second ramification break $u_2 - 2$ over the ideal $(t)$ and inertia group congruent to $\mathbb{Z}/2\mathbb{Z} \times \mathbb{Z}/2\mathbb{Z}$ over the ideal $(t-\varpi)$. By the hypothesis above and~\cite{pagot2002}, the base change of this generic fiber to the algebraic closure lifts to characteristic zero. Thus $L|K$ lifts to characteristic zero by Theorem~\ref{mumfthm}.

Suppose secondly that the sequence of ramification groups of $L$ is of Type II. By Propostition~\ref{cyclicdef}, $L$ admits an equicharacteristic deformation whose generic fiber has second ramification break $u_2 - 2$ over the ideal $(t)$ and inertia group congruent to $\mathbb{Z}/4\mathbb{Z}$ over the ideal $(t-\varpi)$. By the hypothesis above and~\cite{GM98}, the base change of this generic fiber to the algebraic closure lifts to characteristic zero. Thus $L|K$ lifts to characteristic zero by Theorem~\ref{mumfthm}.

Suppose finally that the sequence of ramification groups of $L$ is of Type III. By Proposition~\ref{upperprop}, $u_1 = u_2$. Since $u_2 > 1$ by supposition, $u_1 > 1$ as well. Therefore, $L$ admits an equicharacteristic deformation whose generic fiber has second ramification break $u_2 - 2$ over the ideal $(t)$ and second ramification break 1 over the ideal $(t-\varpi)$ by Proposition~\ref{maindef}. Thus the base change of this generic fiber to the algebraic closure lifts to characteristic zero by hypothesis. Hence $L|K$ lifts to characteristic zero by Theorem~\ref{mumfthm}.
\end{proof}

\subsection{Supersimple Extensions}

The propositions of the previous subsection have effectively reduced the proof that $D_4$ is a local Oort group to showing that every $D_4$-extension of a complete discrete valuation field of characteristic two with algebraically closed residue field whose second upper ramification break is 1 lifts to characteristic zero. 
That all such extensions do, in fact, lift to characteristic zero, is a result of Brewis in~\cite{Brewis2008}, phrased there in somewhat different language.


Let $K$ be a complete discrete valuation field of characteristic two with algebraically closed residue field, and let $L$ be a Galois extension of $K$ such that $\text{Gal}(L|K) \cong D_4$. Following Brewis, we fix $a,b \in D_4$ such that $D_4 = \langle a, b \mid a^4 = b^2 = e, bab^{-1} = a^3 \rangle$.

\begin{defn}[Brewis] 
The extension $L$ over $K$ is \emph{supersimple} if both of the following two conditions hold:

\begin{enumerate}
\item The degree of different of $L^{\langle a^2 \rangle}$ over $L^{\langle a^2, b\rangle}$ is 2.
\item The degree of different of $L^{\langle a^2, b \rangle}$ over $K$ is 2.
\end{enumerate}

\end{defn}

The main result of~\cite{Brewis2008}, denoted therein as Theorem 4, is as follows: 

\begin{thm}[Brewis]\label{brewisthm}
If $L|K$ is supersimple, then $L|K$ lifts to characteristic zero.
\end{thm}

To rephrase Theorem~\ref{brewisthm} in terms of the ramification breaks of $L$ over $K$, we shall need the following proposition.

\begin{prop}\label{equivprop} The extension $L|K$ is supersimple if and only if the second ramification break of $L|K$ is 1.
\end{prop}

\begin{proof} By Hilbert's different formula~\cite{LocalFields}, $L|K$ is supersimple if and only if both the conductor of $L^{\langle a^2 \rangle}$ over $L^{\langle a^2, b\rangle}$ and the conductor of $L^{\langle a^2, b \rangle}$ over $K$ are equal to 1. By Lemma~\ref{raylem1} and Lemma~\ref{raylem2}, this occurs if and only if all three of the conductors over $K$ of the degree two subextensions $L^{\langle a^2, b \rangle}$, $L^{\langle a^2, ab \rangle}$ and $L^{\langle a \rangle}$ are equal to 1. By Proposition~\ref{breakprop}, this occurs if and only if the second ramification break of $L|K$ is 1.
\end{proof}


\begin{cor}\label{brewiscor}
Suppose that the second upper ramification break of $L$ over $K$ is 1. Then $L|K$ lifts to characteristic zero.
\end{cor}

\subsection{Proof of Main Theorem}
We conclude by proving the main theorem of the article and observing an immediate corollary.

\begin{thm}\label{mainthm}
The group $D_4$ is a local Oort group for the prime 2. That is, the following statement holds:

Let $K$ be a complete discrete valuation field of characteristic two with algebraically closed residue field, and let $L$ be a Galois extension of $K$ such that $\emph{Gal}(L|K) \cong D_4$. Then $L|K$ lifts to characteristic zero.
\end{thm}

\begin{proof}
Let $u_2$ denote the second upper ramification break of $L$ over $K$. We shall proceed by strong induction on $u_2$. By Corollary~\ref{breakcor}, $u_2$ is odd. The base case ($u_2 = 1$) is given by Corollary~\ref{brewiscor}. Since $u_2$ is odd, the induction step is given by Proposition~\ref{inductionprop}. Thus $L|K$ lifts to characteristic zero, as claimed.
\end{proof}

By Theorem~\ref{cgh08thm} (or by Theorem~\ref{cgh17thm}), Theorem~\ref{mainthm} implies the following corollary.

\begin{cor}\label{globalcor}
The group $D_4$ is an Oort group for the prime 2.
\end{cor}

\begin{remark} One might hope to use the methods used in this paper to prove that $D_8$, or more ambitiously, $D_{2^n}$ for some $n \ge 4$, is also a local Oort group for $p=2$. However, there are at present at least two substantial obstacles to such a proof.

Firstly, the calculation of the ramification breaks (and hence the differents) of $D_4$-extensions of complete discrete valuation fields presented in Subsection~\ref{rambrkss} depends essentially on the fact that the Galois closure of any non-Galois two-level tower of $\mathbb{Z}/2\mathbb{Z}$-extensions of a field is a $D_4$-extension of that field. While $D_8$-extensions do occur as the Galois closures of smaller field extensions, there is no similarly simple class of extensions whose Galois closures are invariably $D_8$-extensions.  The situation for higher dihedral extensions is similar to that for $D_8$-extensions.

Secondly, the effective use of the `method of equicharacteristic deformation' requires a base case of extensions known to lift to characteristic zero. In the $D_4$ case, the work of Brewis in~\cite{Brewis2008} provided this base. However, neither in the $D_8$ case nor in any higher dihedral case is any extension in characteristic two known to lift to characteristic zero.
\end{remark}

\section*{Acknowledgement}

The author thanks Andrew Obus, his doctoral advisor, for the latter's considerable assistance in guiding, at a multitude of points, both the direction of the work presented in this paper and the writing of the paper itself.

\bibliography{Literature}
\bibliographystyle{alpha}

\end{document}